\documentclass[a4paper,11pt]{article}
\usepackage[letterpaper,margin=0.80in,bottom=0.90in,top=0.85in]{geometry}

\usepackage[english]{babel}

\usepackage[utf8]{inputenc}
\usepackage{amsmath,mathrsfs}
\usepackage{indentfirst}
\usepackage{fancyhdr}
\usepackage{amssymb}
\usepackage{amsthm}
\usepackage[dvips]{graphicx}
\usepackage{color}
\usepackage{cancel}
\usepackage{subcaption}
\usepackage[dvipsnames]{xcolor}
\usepackage{eucal}
\usepackage{latexsym}
\usepackage{chngcntr}
\usepackage{faktor}
\usepackage{microtype}
\usepackage{newlfont}
\usepackage{enumitem}
\usepackage{hyperref,geometry,mathtools} 
\usepackage{todonotes}

\usepackage{amsmath, amsfonts, amsthm,amssymb,  bbm}
\usepackage{tikz-cd}
\usepackage{multicol}
\usepackage{cancel}
\usepackage{mathtools}
\usepackage{graphicx}

\newcommand{\uphi}{\underline{\varphi}}

\newcommand{\umu}{\underline{\mu}}

\newcommand{\betone}[3]{[b_{#1}^{(2)}]_{#2}^{#3}}
\newcommand{\betonet}[3]{\big[b_{#1}^{(3)}\big]_{#2}^{#3}}
\newcommand{\betoneq}[3]{\big[b_{#1}^{(4)}\big]_{#2}^{#3}}

\newcommand{\e}{\epsilon}
\newcommand{\ttf}{\mathtt{f}}
\newcommand{\im}{\mathrm{i}\,}

\newcommand\restr[2]{{% we make the whole thing an ordinary symbol
  \left.\kern-\nulldelimiterspace % automatically resize the bar with \right
  #1 % the function
  \vphantom{\big|} % pretend it's a little taller at normal size
  \right|_{#2} % this is the delimiter
  }}

\allowdisplaybreaks
\newcommand{\ch}{{\mathtt c}_{\mathtt h}}

\theoremstyle{plain}
\newtheorem{lem}{Lemma}
\newtheorem{teo}[lem]{Theorem}

\newtheorem{conj}[lem]{Conjecture}

\theoremstyle{definition}

\newcommand{\uno}{\mathrm{Id}}

\newcommand{\bR}{\mathbb{R}}
\newcommand{\bT}{\mathbb{T}}
\newcommand{\bZ}{\mathbb{Z}}
\newcommand{\bN}{\mathbb{N}}

\newcommand{\bC}{\mathbb{C}}

\newcommand{\cL}{{\cal L}}

\newcommand{\tp}{\mathtt{p}}

\newcommand{\pa}{\partial}

\newcommand{\tth}{\mathtt{h}}

\numberwithin{equation}{section}
\counterwithin{lem}{section}
\title{\bf On higher order isolas of unstable Stokes waves}

\begin{document}

 \author{Massimiliano Berti, Livia Corsi, Alberto Maspero, Paolo Ventura}

\date{}

\maketitle

\noindent 

\noindent
{\bf Abstract.} We overview the recent result \cite[Theorem 1.1]{BCMV} about the high-frequency 
instability  of Stokes waves subject to longitudinal perturbations.  The 
spectral bands of unstable eigenvalues away from the origin form a sequence of
{\it isolas} %(elliptic branches) 
parameterized by an integer $ \tp \geq 2 $ for any value of the depth $ \tth > 0 $ % of the ocean 
such that an explicit  analytic 
function $\beta_1^{(\mathtt{p})}(\tth)  $ is not zero.
In \cite{BCMV} it is proved that the map $ \tth \mapsto \beta_1^{(\mathtt{p})}(\tth)  $ is not identically zero for any 
$ \tp \geq 2 $ 
 by showing that  $ \lim_{\tth \to 0^+}\beta_1^{(\mathtt{p})}(\tth) = - \infty $. 
    In this manuscript  we compute the asymptotic expansion 
    of $\beta_1^{(\mathtt{p})}(\tth) $
    in the deep-water limit $ \tth \to + \infty $ --it vanishes exponentially fast to zero-- 
    for $\tp=2$, $3$, $4$.

% \tableofcontents

\section{Introduction}

Stokes waves are among the most renowned global in time solutions of the gravity water waves equations,
discovered in the pioneering work 
\cite{stokes} of Stokes  in 1847. %  We remind that 
The water waves equations govern 
the time evolution of 
% the free surface $ \eta (t,x) $ 
a bidimensional, incompressible, perfect 
fluid  occupying  the time dependent region 
%We consider the Euler equations for a 2-dimensional  %incompressible 
%and irrotational fluid 
% velocity field % under the action of  gravity 
$$
{ \mathcal D}_\eta := \big\{ (x,y)\in \bT\times \bR\;:\; -\tth < y< \eta(t,x) \big\} \, , 
\quad  \bT :=\bR/2\pi\bZ  \, , 
$$ 
under the action of  gravity. 
The dynamics of the fluid is 
determined  by the Euler equations inside $ {\mathcal D}_\eta $ and 
by two boundary conditions at the free surface. 
The first is that the fluid particles  remain, along the evolution, on the free surface   
(kinematic boundary condition), 
and the second one is that the pressure of the fluid  
is equal, at the free surface, to the constant atmospheric pressure 
 (dynamic boundary condition). 
As shown 
by Zhakarov \cite{Zak1} % and Craig-Sulem \cite{CS} 
these equations form an infinite dimensional Hamiltonian system. 

The Stokes waves 
are space and time 
periodic traveling solutions of the form
$ \eta_\e (x- c_\e t  )$ with a periodic profile 
$ \eta_\e (\cdot ) $.
%thus steady in a reference inertial frame moving with speed $ c_\e $.   
  The first mathematically rigorous proof of their  existence was given by \cite{Struik,LC,Nek} almost one century ago.

\smallskip
 A question of fundamental physical importance concerns their stability/instability. 

\smallskip 
In the sixties 
 Benjamin and Feir \cite{BF,Benjamin}, Whitham \cite{Whitham},  Lighthill \cite{Li} and Zakharov \cite{Z0,ZK} discovered, 
 through experiments and formal arguments, that small amplitude  Stokes waves in sufficiently deep water are  unstable
when subject to long-wave perturbations. More precisely, their works predicted 
that
%the existence of a critical depth 
%$ \tthWB := 1.363... $ such that  $ 
2$\pi $-periodic Stokes waves evolving in a tank of  depth
$   \mathtt h > 1.363... $  and length $  2 \pi  / \mu $, with % Floquet exponent 
$  \mu \ll 1 $, are unstable. 
This phenomenon is nowadays called ``Benjamin-Feir'' --or modulational-- instability.
It is supported by an enormous  amount of  
physical observations and numerical simulations, see e.g. \cite{DO,KDZ},  
and detected in  several dispersive  fluid PDE models, such as KdV, gKdV, NLS and the Whitham equation, 
see e.g. % \cite{Wh,SHCH,GH,BJ,HJ,BHJ,MR}.
\cite{Wh,SHCH,GH,HK,BJ,J,HJ,BHJ,HP, MR}.

In the eighties other mechanisms of instability
-called nowadays high frequency instabilities-
were detected by McLean \cite{Mc1,Mc2}.  
In order to explain properly these phenomena  
we formulate  the 
problem in  rigorous mathematical terms.
\\[1mm]
{\bf Mathematical formulation.}
%  as follows.  
 %It is convenient to use the Zakharov fortmulation, since the Hamiltonian and reversible 
 %structure play a key role. 
 Consider 
% the pure gravity water waves equations  for a bidimensional fluid in an 
   a $2\pi$-periodic  Stokes wave 
    with amplitude $0< \e \ll 1$
    in an ocean with depth $ \tth > 0 $.
The linearized  water waves equations at the Stokes waves
are, in the inertial reference frame moving with the speed $ c_\e $ of the wave,
% =
% \ch + O(\e^2 )$, $ \ch := \sqrt{\tanh(\tth)} $
 a linear time-independent system of partial pseudo-differential equations 
\begin{equation}
\label{linoriginale}
\pa_t h (t,x) = \mathcal{L}_{\e}
( \tth) \, [h(t,x)]  \, .  
\end{equation}
Using the Zakharov Hamiltonian formulation \cite{Zak1} 
of the water waves 
equation (and after some changes of variables) 
the linear operator $\cL_\e (\tth) $ 
%$  H^1(\bT,\bR^2) 
%\to L^2(\bT, \bR^2) $ 
%is a pseudo-differential linear operator  with $ 2 \pi $-periodic coefficients.
%Actually 
%(after conjugating with the ``good unknown of Alinhac" and the ``Levi-Civita" transformations)
% e refer to  \cite{BMV1,BMV2,BMV_ed}, 
%the linear system 
%\begin{equation}\label{linWW}%
%h_t = \cL_\e h   
%\end{equation}
% where 
%$\cL_\e $ 
turns out to be the Hamiltonian and reversible real operator
\begin{equation}\label{cLepsilon}
 \cL_\e(\tth)   
:=  
\begin{bmatrix} 
0 & \uno \\ -\uno & 0 \end{bmatrix}  
\begin{bmatrix}   1+a_\e(x) &   -(\ch+p_\e(x)) \pa_x \\ 
\pa_x \circ (\ch+p_\e(x)) &  |D|\tanh((\tth+\mathtt{f}_\e) |D|)  
\end{bmatrix} 
\end{equation} 
where  
$p_\e (x), a_\e (x) $ are real  even analytic functions and $\ttf_\e $ is a real number.
Both $p_\e (x), a_\e (x) $ and  $\ttf_\e $ depend analytically with respect to  $ \e $. We refer to \cite{NS,BMV1,BMV3} for details.

% \footnote{The operator $ \mathcal{L}_{\e}
% (\tth) $ is given  in \eqref{cLepsilon} and it is   obtained 
% conjugating the linearized water waves equations in the Zakharov  formulation 
% via the  ``good unknown of Alinhac"  and the 
% Levi-Civita  invertible transformations.}.
The dynamics 
of \eqref{linoriginale} for  $ h (t,\cdot) $  in  $ L^2 (\bR) $ is fully determined by  the 
spectrum of $  \mathcal{L}_{\e}
(\tth)  $ in  $ L^2 (\bR) $. 

Let us % describe 
 summarize what is known about it.
% the mathematical state of the art. 
%In recent years it has been proved the following: the spectrum contains a figure ``8" close to the origin \cite{BrM,NS,BMV1,BMV3,BMV_ed} and an isola right above  the figure ``8" \cite{HY,CNS,BMV4}.  
%Supported by numerical simulations,
%it had  been conjectured 
%in \cite{DO}  the
% {\em  existence of infinitely many isolas}
%in the spectrum, 
%centered along the 
%imaginary axis, and  shrinking exponentially fast away from the origin, see \ref{fig:isole}. 
%This has been proved in \cite{BCMV}. 
The first mathematically  rigorous proof 
of local branches of  Benjamin-Feir spectrum was obtained   by Bridges-Mielke \cite{BrM}  in finite depth, for $ \tth > 1.363... $,  see also \cite{HY}, and recently  by Nguyen-Strauss \cite{NS}  
in deep water.   
Very recently  Berti-Maspero-Ventura \cite{BMV1,BMV3,BMV_ed}   proved that these local branches can be globally continued to form a  complete figure ``8'', a fact that   
had been observed numerically in \cite{DO}. The papers \cite{BMV1}, resp. \cite{BMV3},  proved the existence of the figure ``8" in deep water, resp.  in     finite depth  $ \tth > 1.363... $; finally  \cite{BMV_ed} 
proved that also at 
the critical depth $ \tth = 1.363... $, and also slightly below it, the Stokes waves are modulationally unstable, answering to a question   that had been controversial for long times.  
This settles the question about the shape of the spectrum close to the origin.

\smallskip
In the eighties McLean \cite{Mc1,Mc2}
 revealed  
%Preliminary numerical investigations about the 
the existence of unstable  spectrum away from the origin.  Later on 
Deconinck and Oliveras  obtained in \cite{DO}
the complete numerical 
plots of  these high-frequency 
instabilities, which    appear  as small isolas   centered on the imaginary axis. In contrast to Benjamin-Feir instability, they  occur at all values of the depth. 
Deconinck and Oliveras 
% were able to plot the first few of these isolas, and 
conjectured the existence of infinitely many isolas, parametrized by integers $\tp \geq 2 $, going to infinity along the imaginary axis  and shrinking as $ O(\e^\tp) $, see Figure \ref{fig:isole}.

\begin{figure}[h!!!]\centering \subcaptionbox*{}[.45\textwidth]{\includegraphics[width=6cm]{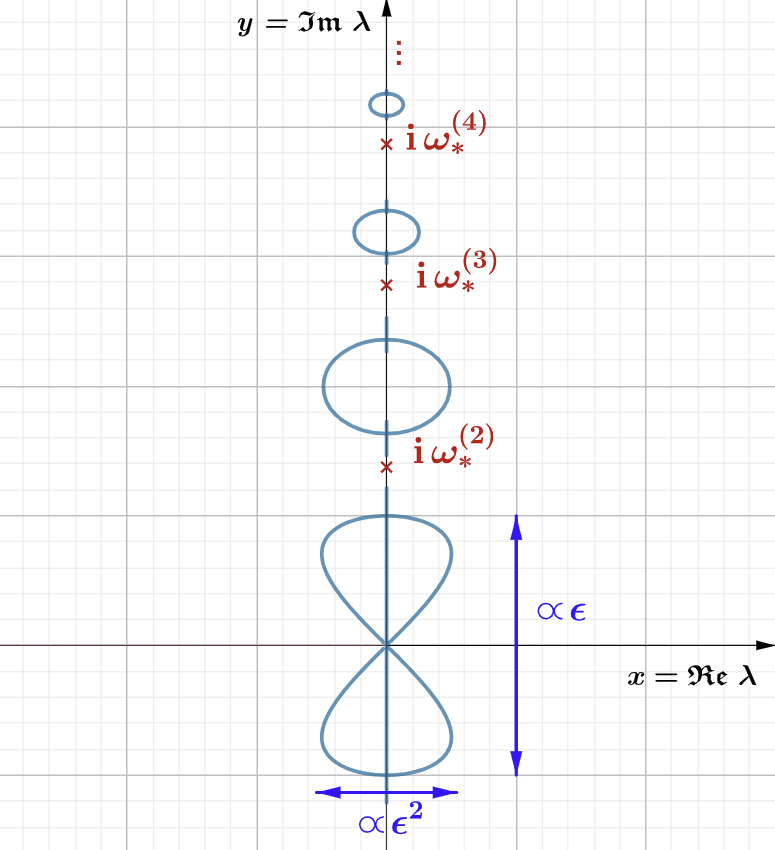}}
\subcaptionbox*{}[.45\textwidth]{
\includegraphics[width=7cm]{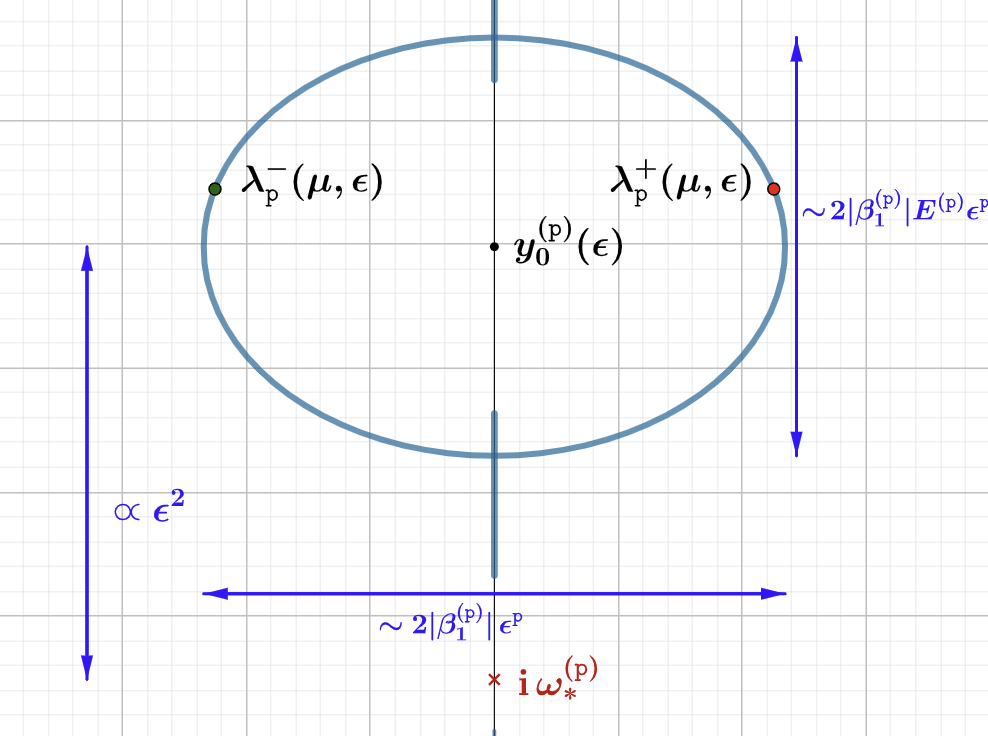} }
 \vspace{-0.8cm}
\caption{ \label{fig:isole} 
Spectral bands with non zero real part of the 
$ L^2 (\bR) $-spectrum of
$ \cL_{\e} $. 
On the right,  zoom of the $\tp$-th isola of modulational instability.  Its center $ y_0^{({\mathtt p} )} (\e) $ is $ O(\e^2) $ distant from $ \im \omega_*^{({\mathtt p} )} $ and its size is $ \propto \e^{{\mathtt p} } $. As shown in Theorem \ref{thm:main2}, two eigenvalues $\lambda^{\pm}_\tp(\mu,\e)$ of the Floquet operator $\cL_{\mu,\e}$  span the $\tp$-th isola for Floquet exponents $\mu \in (\mu_\wedge^{({\mathtt p} )}(\e),\mu_\vee^{({\mathtt p} )}(\e) ) $ and recollide on the imaginary axis at the ends of the interval.  
}
\end{figure}

A formal perturbation method to describe them   has been developed in \cite{CDT} for the first two isolas closest to the origin 
 (i.e. for $\tp=2,3 $ in the language of Theorem \ref{thm:main2}).  
 The first  high-frequency instability isola 
 (i.e. $\tp= 2 $) has been  rigorously proved in  \cite{HY} in finite depth and 
in \cite{BMV4} for the deep-water case. 

\smallskip 

The  proof of the conjecture of Deconinck and Oliveras   is a challenging task  since  
the $\tp$-th isola, if ever exists,  has
exponentially small size $ \propto \e^\tp$, 
it is spanned for an interval of Floquet exponents of very small width $ \propto \e^\tp$ and 
depends on the Taylor expansion of the Stokes waves  at order $\tp$.  
Notwithstanding, % The long standing conjecture of 
the 
existence of infinitely many isolas of instability has been recently proved in \cite{BCMV}, see 
Theorem \ref{thm:main2}.

Before stating precisely Theorem \ref{thm:main2}
and the main novel contribution of this manuscript, 
Theorem \ref{thm:main2new},  we  mention 
the nonlinear Benjaimin-Feir instability result in % Chen-Su 
\cite{ChenSu}, 
the very  recent papers \cite{CNS,CNS1} 
about instability of the Stokes waves 
under tranverse perturbations  
%More precisely, 
%these works prove 
%relying on the spectral approach in 
%\cite{BMV1},   
%the first instability isola for   Stokes waves %under transversal perturbations
in finite and infinite depth, and
%Instability of Stokes waves has been 
%very recently detected  the preprint 
\cite{Sun} for longitudinal and transversal perturbations.
Another  result of transverse instability of 
gravity-capillary Stokes waves is proved in   \cite{HTW}. 
%Instability of Stokes waves has been 
%very recently detected 
% in \cite{Sun} for under longitudinal and transversal perturbations.
\\[1mm]
{\bf Bloch-Floquet expansion.}
Since the operator $\cL_\e := \cL_\e(\tth)$ in 
\eqref{cLepsilon} has $2\pi$-periodic coefficients, its $ L^2(\bR) $ spectrum is conveniently described 
by the Bloch-Floquet decomposition 
 \begin{equation}
     \label{BFtheory}
     \sigma_{L^2(\bR)}\big( 
     \cL_\e  \big)
 = \bigcup_{\mu \in \big[-\tfrac12, \tfrac12 \big)}
 \sigma_{L^2(\bT)}\big( \cL_{\mu,\e} 
 \big)\, \qquad \textup{where}\qquad 
 \cL_{\mu,\e} :=
\cL_{\mu,\e} (\tth) := e^{-\im \mu x}\, \cL_\e\,  e^{\im \mu x}  
 \end{equation}
is 
%a pseudo-differential  
%operator,
%with $2 \pi $ coefficients, 
%depending on $    \mu $. 
%It turns out to that $  \cL_{\mu,\e} $ is 
the complex  \emph{Hamiltonian} and \emph{reversible} pseudo-differential  operator
\begin{align}\label{WW}
 \cL_{\mu,\e}    
 %:= \begin{bmatrix} (\pa_x+\im\mu)\circ (\ch+p_\e(x)) & 
 % |D+\mu| \tanh\big((\tth + \ttf_\e) |D+\mu| \big) \\ %-(1+a_\e(x)) & (\ch+p_\e(x))(\pa_x+\im \mu) %\end{bmatrix} \\ 
 = \begin{bmatrix} 0 & \uno\\ -\uno & 0 \end{bmatrix} \begin{bmatrix} 1+a_\e(x) & -(\ch+p_\e(x))(\pa_x+\im \mu) \\ (\pa_x+\im\mu)\circ (\ch+p_\e(x)) & |D+\mu| \tanh\big((\tth + \ttf_\e) |D+\mu| \big) 
 \end{bmatrix}   \, ,   
\end{align} 
% with $2 \pi $ coefficients.
with 
domain $H^1(\bT):= H^1(\mathbb{T},\bC^2)$ and range $L^2(\bT):=L^2(\mathbb{T},\bC^2)$.

The spectrum  of
$ \cL_{\mu,\e} $ 
on 
$ L^2 (\bT ) $ 
is discrete and its eigenvalues  span, as $ \mu $ varies,   the continuous spectral  bands of $\sigma_{L^2(\bR)}\big( 
     \cL_\e  \big)$. 
The parameter 
$ \mu $ is usually referred to as the Floquet exponent.
A direct consequence of the Bloch-Floquet decomposition is that any 
solution 
of \eqref{linoriginale} can be decomposed as a linear superposition of Bloch-waves
\begin{equation}
\label{hesplode}
h(t,x) =
e^{ \lambda t}
e^{\im \mu x} v(x)
\end{equation}
where  $\lambda
$ is an eigenvalue of $\cL_{\mu,\e} $ 
with  associated  eigenvector 
$ v(x) \in L^2(\bT) $. 
 If $\lambda $ is  unstable, i.e. has positive real part,  then the 
solution 
\eqref{hesplode}
 grows exponentially fast in time.  
 
 %Indeed, Bloch-Floquet theory ensures that each spectral value $\lambda \in \sigma_{L^2(\bR)}\big( \cL_\e(\tth) \big)$ is an  eigenvalue 
 % of the operator  $ \mathcal{L}_{\mu,\e} 
%:= e^{-\im \mu x } \,\mathcal{L}_{\e} \, e^{\im \mu x } $ in \eqref{WW}
%acting on $2\pi$-periodic functions. \\ 

We remark that the spectrum
$ \sigma_{L^2(\bT)}(\cL_{\mu,\e}) $
is a set which is
$1 $-periodic in $ \mu $, and thus
it is sufficient to consider 
$ \mu $ in  the first zone of Brillouin $ [-\tfrac12, \tfrac12) $. Furthermore by the reality 
of $ \cL_{\e} $ the spectrum $ \sigma_{L^2(\bT)}(\cL_{-\mu,\e}) $
is the complex conjugated of 
$ \sigma_{L^2(\bT)}(\cL_{\mu,\e})$
and we can restrict   $\mu $ to the interval $ [0,\tfrac12) $. From a physical point of view, the spectrum of the operator   $\cL_{\mu,\e}$ encodes the linear dynamics of any long-wave perturbation of the Stokes waves of amplitude $\e $ 
evolving inside a tank of length $ 2 \pi \mu^{-1} $. 

\smallskip
The Hamiltonian nature of the operator $\cL_{\mu,\e}  $ constrains the eigenvalues with nonzero real part to arise, for
small $ \e > 0  $, as perturbation of {\it multiple}  eigenvalues of 
% the Fourier multiplier matrix operator  
% $\cL_{\mu,0} (\tth) $. 
%The spectrum of  the  Fourier multiplier matrix operator 
\begin{equation}\label{cLmu}
 \cL_{\mu,0} = 
 \begin{bmatrix} 
 \ch (\pa_x+\im\mu)  & |D+\mu|\tanh(\tth |D+\mu|) \\ -1 & \ch(\pa_x+\im\mu) \end{bmatrix} \, .  
\end{equation}
The spectrum of the 
Fourier multiplier matrix operator  
 $\cL_{\mu,0} (\tth) $ 
on $L^2(\bT,\bC^2)$  
is given by the purely imaginary numbers
\begin{equation}\label{omeghino}
\lambda_j^\sigma(\mu,\tth):=
\im \omega^\sigma(j+\mu,\tth)=   \ch(j+\mu) -\sigma \Omega(j+\mu, \tth)   \, ,
\  \forall \sigma   = \pm \, , \,  j\in \bZ \, , 
\end{equation}
where  
\begin{equation}\label{Oomegino}
\omega^\sigma(\varphi,\tth) :=  \ch \varphi -\sigma \Omega(\varphi,\tth) \, ,  \qquad 
 \Omega(\varphi,\tth) :=\sqrt{\varphi\tanh(\tth \varphi)}  \, , 
\qquad   \varphi \in \bR\,  ,\ \ \sigma=\pm\, .
\end{equation}
Note that $\omega^+(-\varphi,\tth) = -\omega^-(\varphi,\tth) $. Sometimes   
$ \sigma $ is called the Krein signature of the eigenvalue $ \lambda_j^\sigma(\mu,\tth) $.
The spectrum of $\cL_{\mu,0}(\tth) $ is  purely imaginary. 
As stated in \cite{BCMV}[Lemma 2.9], for any 
$\tth > 0 $ and $ \mu \in \bR$, the operator $\cL_{\mu,0} (\tth) $ possesses, away from $0$, only simple or double eigenvalues. Its   double  eigenvalues
form a diverging sequence  
$$ 
\{\pm \im \omega_*^{(\tp)}(\tth ) \}_{\tp=2,3,\dots} $$ 
 where each $ \pm \im \omega_*^{(\tp)}(\tth)  $ lies in the spectrum of the operator $\cL_{\umu,0}(\tth) $ only for {\it one} particular real value of the Floquet exponent $$
 \umu=\uphi(\tp,\tth) > 0  
 $$
 and its integer  shifts $ \umu +k$, 
 $k\in \bZ$, 
 see \cite{BCMV}[Lemma 2.7].
 In other words at $ \umu = \underline{\varphi} (\tp,\tth) $,
 mod $ 1$, the unperturbed spectral bands of $ \cL_{\mu,0} 
 (\tth) $ intersect. 
 The values $ \pm \im \omega_*^{(\tp)}(\tth) $ are the branching points of the $\tp$-th isolas stated in Theorem \ref{thm:main2new} below.  
We introduce a notation.

\begin{itemize} \item  {\bf Notation}.
  We denote by $r(\nu^{m_1}\e^{n_1},\dots,\nu^{m_q}\e^{n_q}) $  
  a 
scalar  real analytic function  that  satisfies,  for some $ C > 0 $ and $(\nu, \e)$ small, the bound
 $ |  r (\nu^{m_1}\e^{n_1},\dots,\nu^{m_q}\e^{n_q}) |   \leq C \sum_{j = 1}^q|\nu |^{m_j}|\e|^{n_j}\, 
 . $ For any  $\delta > 0 $, we denote by $B_\delta(x)$ the real interval $(x-\delta,x+\delta)$ centered at $x$.
\end{itemize}

The  next result gives a precise description of the  spectrum  near $\im \omega_*^{(\tp)}(\tth)$ of the operator $ \cL_{\mu,\e} (\tth) $,  
for any $(\mu, \e)$ sufficiently close to $(\umu,0)$, value  at which
$\cL_{\umu,0}(\tth) $ has a double eigenvalue.

\begin{teo}\label{thm:main2} 
{\bf (Unstable spectrum \cite{BCMV})} For any integer $\tp \in \bN $, $\tp  \geq 2$, for any $ \tth >0 $ let $\umu
= \uphi(\tp,\tth) > 0 $  such that the operator $ \cL_{\umu,0}(\tth)$ in \eqref{WW} has a 
double eigenvalue at $\im \omega^{(\tp)}_* (\tth) $.   
Then there exist
$\e_1^{(\tp)}, \delta_0^{(\tp)}>0$ and real  analytic functions $\mu_0^{(\tp)},\mu_\vee^{(\tp)},\mu_\wedge^{(\tp)} \colon [0, \e_1^{(\tp)}) \to B_{\delta_0^{(\tp)}}(\umu)$  satisfying 
$$
\mu_0^{(\tp)} (0) =  \mu_{\wedge}^{(\tp)}(0) = \mu_\vee^{(\tp)}(0) =\umu
$$ 
and, for any 
$ 0 < \e < \e_1^{(\tp)}   $, 
 \begin{equation}\label{mupm}
  \mu_\wedge^{(\tp)}(\e) < \mu_0^{(\tp)}(\e) < \mu_\vee^{(\tp)}(\e)\, ,\quad  |\mu_{\wedge,\vee}^{(\tp)}(\e)- \mu_0^{(\tp)}(\e) |   =   \frac{2 |\beta_1^{(\tp)}(\tth) |}{T_1^{(\tp)} (\tth) } \e^\tp+ r(\e^{\tp+1})\, ,
 \end{equation}
(see Figure \ref{fig:instareg}) where   $T_1^{(\tp)}(\tth) > 0 $ and $\beta_1^{(\tp)}(\tth) \not\equiv 0 $ are real 
analytic functions defined for any
$ \tth > 0 $,      such that the following holds true: 

for any $(\mu, \e) \in B_{\delta_0^{(\tp)}}(\umu) \times B_{\e_1^{(\tp)}}(0)$ 
the operator $ \cL_{\mu,\e} (\tth) $ possesses a pair of eigenvalues  of the form  
\begin{align}\label{final.eig}
&\lambda^\pm_\tp \big(\mu, \e \big) 
=\begin{cases} \im \omega_*^{(\tp)}+ \im s^{(\tp)}(\mu,\e)  \pm \dfrac12 \sqrt{ D^{(\tp)}\big(\mu,\e\big)}  &\mbox{if }  \mu \in \big( \mu_\wedge^{(\tp)}(\e) , \mu_\vee^{(\tp)} (\e)\big) \, , \\[3mm]
\im \omega_*^{(\tp)}+ \im  s^{(\tp)}(\mu,\e) 
 \pm \im\sqrt{ \big|D^{(\tp)}\big(\mu,\e\big)\big| } &\mbox{if } \mu \notin \big( \mu_\wedge^{(\tp)} (\e) , \mu_\vee^{(\tp)} (\e)\big)\, ,
\end{cases} 
\end{align} %at  $(\mu,\e) = (\mu_0^{(\tp)}(\e) + \nu, \e)$
 where $s^{(\tp)}(\mu,\e)$, $D^{(\tp)}(\mu,\e)$ are  real-analytic functions of the form 
\begin{equation}\label{degenerateexpD0}
\begin{aligned}
& s^{(\tp)}(\mu_0^{(\tp)} (\e)+\nu,\e) = r(\e^2,\nu)\, , \\
& D^{(\tp)}(\mu_0^{(\tp)} (\e)+\nu,\e) =  4(\beta_1^{(\tp)}(\tth))^2 \e^{2\tp} -  (T_1^{(\tp)}(\tth))^2\nu^2+r(\e^{2\tp+1},\nu\e^{2\tp},\nu^2\e,\nu^3)\, .
\end{aligned}
\end{equation}
 For any $ \tth > 0 $ such that 
 $ \beta_1^{(\tp)} (\tth) \neq 0 $  the function 
$ D^{(\tp)}(\mu_0^{(\tp)} (\e)+\nu,\e) $
is positive in $\big(\mu_\wedge^{(\tp)} (\e), \mu_\vee^{(\tp)} (\e)\big)$, vanishes at $\mu =\mu_{\wedge,\vee}^{(\tp)} (\e) $ and is negative outside  $\big(\mu_\wedge^{(\tp)} (\e), \mu_\vee^{(\tp)} (\e)\big)$. 
For any  fixed $\e \in (0,\e_1^{(\tp)})$ as $\mu$ varies in $\big(\mu_\wedge^{(\tp)}(\e),\mu_\vee^{(\tp)}(\e)\big)$
 the pair of unstable eigenvalues $\lambda^\pm_\tp(\mu,\e)$ in \eqref{final.eig} describes 
 %the $\tp$-th positive isola of $\sigma_{L^2(\bR)}(\cL_{\e})$, 
 %which is  
 a closed analytic  curve 
 in the complex  plane 
 $ \lambda = x + \im y $
 that
  intersects orthogonally the imaginary axis,  
  encircles a convex region, and it is symmetric 
  with respect to $y $-axis.
  Such isola is
 approximated by an  ellipse
 \begin{equation}\label{isolaintro}
x^2  + 
E^{(\tp)}(\tth)^2(1+r(\e^2))(y-y_0^{(\tp)}(\e))^2 =   \beta_1^{(\tp)}(\tth)^2 \e^{2\tp} (1+r(\e)) 
\end{equation}
where $ E^{(\tp)}(\tth) 
\in (0,1) $ is
a real analytic function  for any  
$ \tth > 0 $
and $ y_0^{(\tp)} (\e) $
is $ O(\e^2) $-close to $ \omega_*^{(\tp)} (\tth) $.  
 \end{teo}

Let us make some comments.

 \begin{figure}[h!!]
\centering
\includegraphics[width=7cm]{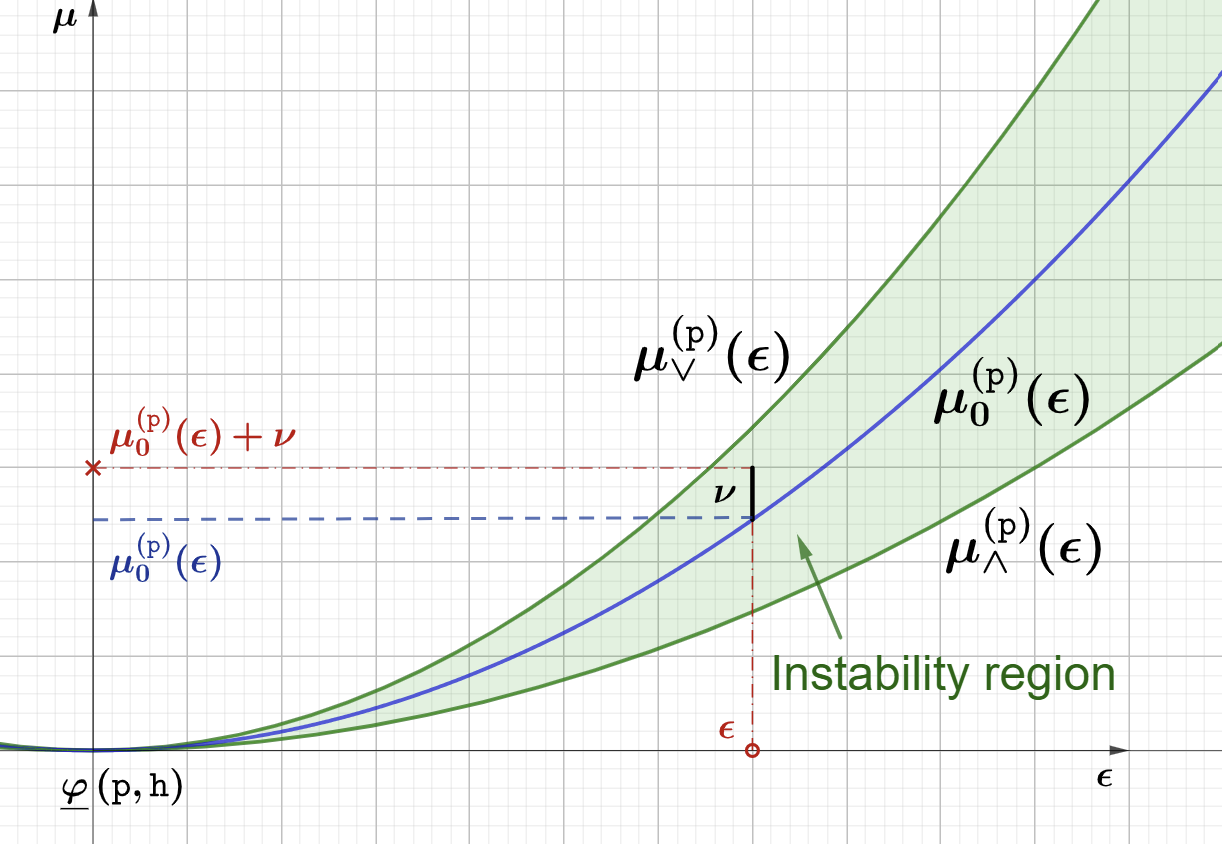}
\caption{ \label{fig:instareg} The instability region around the curve $\mu_0^{(\tp)}(\e)$  delimited by  the curves 
$\mu_\wedge^{(\tp)}(\e)$ and $\mu_\vee^{(\tp)}(\e)$.}
\end{figure}

\begin{enumerate}
\item {\bf Upper bounds:} 
In view of \eqref{degenerateexpD0} and \eqref{mupm}, for any $\tp \geq 2$ and  any depth $\tth>0$,
the real part of the eigenvalues
\eqref{final.eig}
 is at most of size $ O(\e^\tp) $
implying  that the isolas, if ever exist, 
 shrink exponentially fast as 
 $ \tp  \to + \infty $.  
\item 
{\bf 
Lower bounds:} 
In view of
\eqref{final.eig}, \eqref{degenerateexpD0}, 
for any $ \tp \geq 2$ and  any depth $\tth>0$, a sufficient condition for the existence 
of the $\tp$-th instability isola  is that 
$$
\beta_1^{(\tp)}(\tth) \neq 0 \, . 
$$
In such a case
the real part of the eigenvalues
\eqref{final.eig}
 is  of size $  |\beta_1^{(\tp)} (\tth)| \e^\tp+ O(\e^{\tp+1} )$. 
 %The most unstable eigenvalue is reached at $\nu \propto\!\e^{2\tp}$. 
 Furthermore, by \eqref{mupm},  the portion of the 
 unstable band is parametrized by Floquet exponents 
in  the interval
 $ \big( \mu_{\wedge}^{(\tp)}(\e),\mu_{\vee}^{(\tp)}(\e)\big) $ which has exponentially small
 width 
 $ \sim  \frac{4|\beta_1^{(\tp)}(\tth) |}{T_1^{(\tp)}} 
 \e^{\tp} $. 
\item
{\bf 
The function $\beta_1^{(\tp)}(\tth) $:}
It is % analytically 
proved in 
\cite[Proposition 6.6]{BCMV} that, 
for {\it any} $\tp\geq 2 $, the map 
$\tth \mapsto \beta^{(\tp)}_1(\tth)$  is  {\em real analytic} and  
\begin{equation} \label{limbetap}
\lim_{\tth\to 0^+}
\beta_1^{(\tp)}
(\tth) = - \infty \, , 
\qquad
 \qquad  
 \lim_{\tth\to+ \infty}
\beta_1^{(\tp)}
(\tth) = 0 \, , 
\end{equation}
implying, in particular,  that  $  \beta_1^{(\tp)}
(\tth) $  is  not 
identically zero.
For $\tp= 2,3,4 $ the graphs of 
$ \beta_1^{(\tp)}
(\tth) $ are  numerically plotted   in Figures 
\ref{plotb1p2}-\ref{plotb1p4}.
The challenging proof in  \cite{BCMV} 
that
$ \lim_{\tth\to 0^+}
\beta_1^{(\tp)}
(\tth) = - \infty $, for any 
$ \tp \geq 2 $,  is accomplished 
exhibiting the {\it second-order}  
asymptotic expansion of  
$ \beta_1^{(\tp)} (\tth) $ as $ \tth \to 0^+ $.  
The fact that $ \lim_{\tth\to+ \infty}
\beta_1^{(\tp)}
(\tth) = 0 $ 
follows by subtle cancellations in 
the deep water limit. Both limits 
ultimately stem  from  
 astonishing 
structural properties of the Taylor-Fourier coefficients 
% $ \eta_{\ell}^{[\ell]},  \psi_{\ell}^{[\ell]} $
of the Stokes waves 
at any order proved in \cite{BCMV}. 
\begin{figure}[h!]
 \centering
\includegraphics[width=7cm]{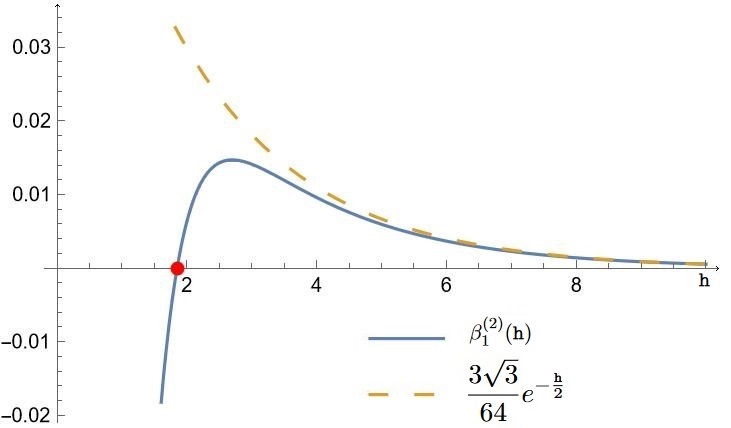}
\caption{ \label{plotb1p2} Case $\tp= 2 $. 
Plot of the function $\beta_1^{(2)}(\tth)$  (in blue), vanishing
at the red point 
$\tth_*\approx 1.84940 $. %see \cite{CDT,HY}
In yellow the leading part of its asymptotic expansion in the deep-water limit.}
\end{figure}
\begin{figure}[h!]
 \centering
\includegraphics[width=7cm]{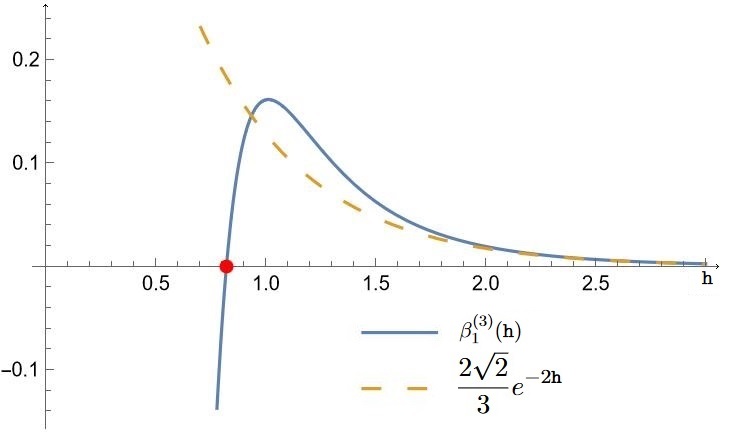}
\caption{ \label{plotb1p3} Case $\tp=3$. Plot of the function $\beta_1^{(3)}(\tth)$ (in blue), vanishing at the red point  $\tth_*\approx 0.82064 $. 
%coherently with  \cite{CDT}
In yellow  the leading part of its asymptotic expansion in the deep-water limit.}
\end{figure}
\begin{figure}[h!!!]\centering \subcaptionbox*{}[.45\textwidth]{\includegraphics[width=4cm]{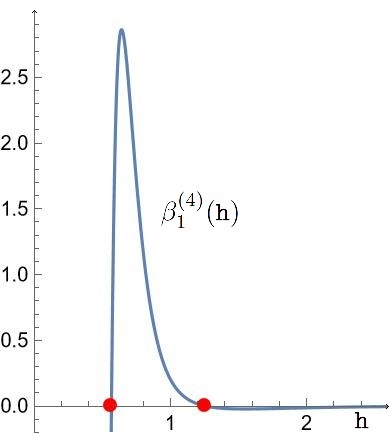}}
\subcaptionbox*{}[.45\textwidth]{
\includegraphics[width=7cm]{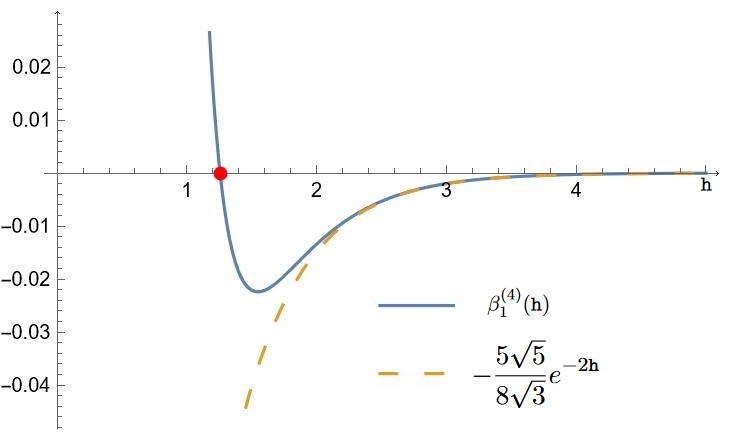} }
 \vspace{-0.8cm}
\caption{\label{plotb1p4} Case $\tp= 4 $. On the left: plot of the function $\beta_1^{(4)}(\tth)$  (in blue) vanishing at the red points $\tth_*\approx 0.566633$  and $\tth_{**}\approx 1.255969$. On the right: zoom on the asymptotic behavior of $\beta_1^{(4)}(\tth)$  in blue, in yellow  the leading part of its expansion in the deep-water limit.
}
\end{figure}

\item
{\bf 
Critical depths:} 
%For any depth $\tth $ outside %$(\beta_1^{(\ell)})^{-1}(0) $,  where 
%$\ell = 2, \ldots, \tp $,
%the $\ell $-th  isola of instability  
%exists and has   
%elliptical shape. 
If  $ \beta_1^{(\tp)} (\tth) = 0 $ 
(this may happen only at a closed set of isolated points) 
it is necessary to further  expand the discriminant $D^{(\tp)}(\mu,\e)$
in \eqref{degenerateexpD0}
to determine its sign. The latter degeneracy occurs
at  
$ \tth = + \infty $ for {\em every} $\tp \geq 2 $ since  $ \lim_{\tth \to + \infty}
\beta_1^{(\tp)} (\tth) = 0 $.
%of the Stokes waves
%coefficients at any order proved in \cite{BCMV}, 
%namely that $ \eta_{\ell}^{[\ell]} = \psi_{\ell}^{[\ell]} $ for any $ \ell \in \bN $. 
For the isola $\tp = 2 $  the higher order expansion required to detect the unstable isola 
in infinite depth
is performed 
in \cite{BMV4}.
\end{enumerate}

\section{Main result}

In this paper we  complement  the above results
showing the behaviour of the function 
$ \beta_1^{(\tp)} (\tth) $ 
for $ \tth \to + \infty $ in the cases 
$ \tp = 2,3,4 $.  
By \eqref{limbetap} 
we  know that
$$
\lim_{\tth \to +\infty}\beta_1^{(\mathtt{p})}(\tth) = 0 \, , \qquad \forall \mathtt{p} \geq 2 \, . 
$$  
Actually 
$ \beta_1^{(\mathtt{p})}(\tth) $ tends to zero exponentially fast. It is expected to converge to zero without oscillations as $ \tth \to +\infty $.
This would imply the Conjecture \ref{conjinf} stated below, that we prove for $ \mathtt{p} =2,3,4 $.

Before stating the main result, let us introduce the following standard notation.  \smallskip

\noindent{\bf Notation}. Let $f,g:(0,+\infty)\to \bR $
be two real functions with $f$ strictly positive, and let $x_0 \in[0,+\infty]$. We denote 
\begin{itemize}
    \item $g=O(f)$ as $x\to x_0 $ if there exists $C>0$ such that $|g(x)| \leq C f(x) $ for any $x$ near $x_0$ ;
    \item $g=o(f)$ as $x\to x_0 $ if $\lim\limits_{x\to x_0} \dfrac{g(x)}{f(x)} = 0 $ .
\end{itemize}

The next result provides the sharp asymptotic expansion of  $\beta_1^{(\mathtt{p})}(\tth) $ as $\tth\to +\infty$ in the cases $\mathtt{p}=2, 3, 4  $. 
\begin{teo} \label{thm:main2new} 
The coefficients $\beta_1^{(p)} (\tth) $, 
$ \mathtt{p} = 2,3, 4 $ admit,
as $ \tth\to +\infty $, 
the  asymptotic expansions 
\begin{align}
& \beta_1^{(2)}(\tth) = 
\frac{3 \sqrt{3} }{64} e^{-\frac{\tth}{2}}+O\big(e^{-\frac{3}{4}\tth} \big) \, , \label{expbeta1p=2}  \\
& \label{expbeta1p=3}
\beta_1^{(3)} (\tth) = 
\frac{2 \sqrt{2} }{3} e^{-2\tth} + O\big(e^{-3\tth}\big) \, ,  \\
& \label{expbeta1p=4}
\beta_1^{(4)}(\tth)  = 
-\frac{5\sqrt{5}}{8\sqrt{3}}   e^{-2\tth} + O\big(e^{-4\tth}\big) \, . 
 \end{align}
\end{teo}

The latter result provides another proof that 
the analytic function $ \tth \mapsto \beta_1^{(\mathtt{p})} (\tth)  $
is not identically zero if $ \mathtt{p} = 2,3,4$. Moreover it rules out the 
existence of infinitely many zeros of  $\beta_1^{(\tp)} (\tth) $ accumulating at 
% one of the extrema of its interval of definition $(0,+
$  \infty$. 
%\begin{itemize}
%  \item for any integer $\tp\geq 2$ as $\tth\to0^+$, see \cite[Theorem 6.1]{BCMV};
%\item for $\tp = 2$, $3$ and $4$ as $\tth\to +\infty$, thanks to Theorem \ref{thm:main2new}.
%\end{itemize}
Since $\beta_1^{(\tp)} (\tth ) $ depends analytically on $\tth$, we deduce that $\beta_1^{(\tp)}(\tth) $ vanishes only at {\it finitely} many values of the depth $ \tth > 0 $ for $\tp=2,3,4$.
We make the following  conjecture. 
\begin{conj}\label{conjinf} For any integer $\tp\geq 2$ the function  $\tth \mapsto \beta_1^{(\mathtt{p})} (\tth)$ has a finite number of zeros in $(0,+\infty) $.
\end{conj}

The expansion \eqref{expbeta1p=2} for the case $ \mathtt{p}=2$ 
could be handled analytically, see Section \ref{sec2}. However 
the computations  
dramatically increase in complexity 
for larger  $ \mathtt{p} $, as evident 
already from 
the case $ \mathtt{p} = 4 $ that we analyze in Section \ref{sec:4}.
We perform these computations with the help of Mathematica, see the notebooks at the page \url{https://git-scm.sissa.it/amaspero/on-higher-order-isolas-of-unstable-stokes-waves/-/tree/main}.
For  $ \tp > 4 $ the problem is analytically open. 

In the next sections we  prove Theorem \ref{thm:main2new}.

\subsection{Case $ \mathtt{p} =2 $}\label{sec2}

We now prove the expansion  \eqref{expbeta1p=2}. 
According to \cite[(5.7)-(5.8)]{BCMV}, the coefficient $\beta_1^{(2)} (\tth) $  is given  by 
 \begin{align}
\label{splittingb1p2}
\beta_1^{(2)}(\tth) &=b_0^{(2)}
 -  \betone{1}{1}{-}
  + \betone{1}{1}{+}
\end{align}
where 
\begin{subequations}\label{b11b12b13}
\begin{align}b_0^{(2)} & :=\frac14 \sqrt{\Omega_0^{(2)}\Omega_{2}^{(2)}} \big(a_2^{[2]}   + p_2^{[2]}(t_0^{(2)}-t_2^{(2)}) \big)\, ,\label{primob1} \\  
 \label{secondob1} 
\betone{1}{1}{-} & := \frac{\Omega_{1}^{(2)}\sqrt{\Omega_0^{(2)}\Omega_{2}^{(2)}}}{16(\ch+\Omega_{1}^{(2)}-\Omega_0^{(2)})}  \big(a_1^{[1]} +p_1^{[1]} (t_{1}^{(2)} - t_{2}^{(2)})  \big) \big(a_1^{[1]}+p_1^{[1]} (t_{1}^{(2)} +  t_0^{(2)} ) \big)\, , \\
 \label{terzob1}
\betone{1}{1}{+}  & := \frac{\Omega_{1}^{(2)}\sqrt{\Omega_0^{(2)}\Omega_{2}^{(2)}}}{16(\ch-\Omega_{1}^{(2)}-\Omega_0^{(2)})}  \big(a_1^{[1]} -p_1^{[1]} (t_{1}^{(2)} +  t_{2}^{(2)})  \big) \big(a_1^{[1]}+ p_1^{[1]}( t_0^{(2)}  -t_{1}^{(2)}) \big)\,,
 \end{align}
 \end{subequations}
with, in view of \cite[(2.20)-(2.23)]{BMV3}, 
\begin{equation}
\label{ap12}
p_1^{[1]} :=  - 2 \ch^{-1}\, ,\quad
  p_2^{[2]} := - \frac{3+\ch^4}{2\ch^7}\, ,\quad
a_1^{[1]} := - ( \ch^2 + \ch^{-2})\, , 
\quad a_2^{[2]} := 
 \frac{-14\ch^4+9\ch^8-3}{4\ch^8}\, , 
\end{equation} 
and, recalling \cite[(5.6)]{BCMV}, for $ j = 0,1,2, $
\begin{equation} \label{tj2}
  \Omega_j^{(2)} := \sqrt{\big(j+\uphi(2,\tth)\big)\tanh\Big(\tth\big(j+\uphi(2,\tth)\big)\Big)}
  \, ,\qquad 
  t_j^{(2)} = \sqrt{\frac{j+\uphi(2,\tth)}{\tanh\Big(\tth\big(j+\uphi(2,\tth)\big)\Big)}}\,  ,
  \end{equation}
%\begin{equation} \label{tj2}
%\begin{alignedat}{2}
%&\Omega_0^{(2)} := \sqrt{\uphi(2,\tth)\tanh(\tth\uphi(2,\tth))}
 % \, ,\qquad &&t_0^{(2)} := \sqrt{\frac{\uphi(2,\tth)}{\tanh(\tth\uphi(2,\tth))}}\, , \\ 
 % &\Omega_1^{(2)} := \sqrt{\big(1+\uphi(2,\tth)\big)\tanh\Big(\tth\big(1+\uphi(2,\tth)\big)\Big)}
 % \, ,\quad  &&t_1^{(2)} = \sqrt{\frac{1+\uphi(2,\tth)}%{\tanh\Big(\tth\big(1+\uphi(2,\tth)\big)\Big)}}\, , \\ 
  %&\Omega_2^{(2)} := \sqrt{\big(2+\uphi(2,\tth)\big)\tanh\Big(\tth\big(2+\uphi(2,\tth)\big)\Big)}
%  \, ,\quad &&t_2^{(2)} := \sqrt{\frac{2+\uphi(2,\tth)}{\tanh\Big(\tth\big(2+\uphi(2,\tth)\big)\Big)}}\, ,
%  \end{alignedat}
%  \end{equation}
  where the wavenumber $\uphi(2,\tth)$ is defined in  \cite[Lemma 2.7]{BCMV}. 

\smallskip 
We now provide the asymptotic expansion 
of the above quantities
as $ \tth \to + \infty $. 
In this case $ \tp = 2 $ the computations are relatively easy and we perform all of them analytically.  We make  use of the asymptotic expansions
\begin{equation}\label{tanhasymp}
\tanh(M) = 1-2e^{-2M} + O\big(e^{-4M} \big)
\quad \textup{as }M\to +\infty\,, \qquad \tanh(\varepsilon) = \varepsilon+O\big(\varepsilon^3\big)  \quad \textup{as }\varepsilon\to 0\,.
\end{equation}
Since 
\begin{equation}\label{dwch}
\ch = \sqrt{\tanh(\tth)} \stackrel{\eqref{tanhasymp}}{=} \sqrt{1-2e^{-2\tth}\big(1+O\big(e^{-2\tth}\big)\big)} = 
1-e^{-2\tth}+O\big(e^{-4\tth} \big)\,,
%= 1 +O\big(e^{-\frac{3}{4}\tth} \big)\, , 
\end{equation}
the terms in \eqref{ap12} admit the  expansions, as $\tth\to +\infty$,
\begin{align} \notag
p_1^{[1]} &=  - 2  - 2 e^{-2\tth}\big(1+O\big(e^{-2\tth}\big)\big) \,,
%=  - 2  +O\big(e^{-\frac{3}{4}\tth} \big)\, ,
\quad
  p_2^{[2]}= - 2 - 12  e^{-2\tth}\big(1+O\big(e^{-2\tth}\big)\big)\,,
  %= - 2 +O\big(e^{-\frac{3}{4}\tth} \big)\, ,
  \\ \label{dwap12}
a_1^{[1]} &= -2 + O\big(e^{-4\tth}\big) \,,
%= -2 +O\big(e^{-\frac{3}{4}\tth} \big)  \, , 
\qquad a_2^{[2]} = 
 - 2 - 20  e^{-2\tth}\big(1+O\big(e^{-2\tth}\big)\big)\,.
 %= - 2 +O\big(e^{-\frac{3}{4}\tth} \big)\, .
 \end{align}
For the terms in \eqref{tj2} we first need a slight improvement of formula \cite[(2.31)]{BCMV}.
\begin{lem} The  wavenumber $\uphi(2,\tth)$ 
defined in {\em \cite[Lemma 2.7]{BCMV}} 
satisfies 
\begin{equation}\label{expuphi}
\uphi(2,\tth) = \frac14 + \frac38 e^{-\frac{\tth}{2}}  \big( 1 +o\big(\tth^{-1}e^{-\frac{\tth}{4}}\big) \big) 
\quad \text{as} \quad \tth \to +\infty \, .
\end{equation}
\end{lem}

\begin{proof}
We recall from \cite[Appendix A]{BCMV} that  
\begin{equation}\label{y2}
\uphi(2,\tth) = \frac14 + y_2(\tth) \qquad \text{where} \qquad 
    y_2(\tth) = \frac38 e^{-\frac{\tth}{2}} \big(1+o(1) \big)
\end{equation} satisfies the equation, see \cite[(A.12)]{BCMV},
\begin{equation}\label{p=2as}
\frac43 y_2 (\tth) = 
\sqrt{  \frac14 + y_2 (\tth) } \, 
e^{ -   \frac{\tth}{2}  }  e^{ - 2 \tth  y_2 (\tth) } 
 + O(y_2^2 (\tth))  + O( e^{-  \tth}) \, .
\end{equation}
We observe that $e^{ - 2 \tth  y_2 (\tth) }  = 1+o\big(\tth^{-1}e^{-\frac{\tth}{4}}\big)$, since
$$
\lim_{\tth\to+\infty}\big( e^{ - 2 \tth  y_2 (\tth) } - 1\big)\tth e^{\frac{\tth}{4}} = \lim_{\tth\to+\infty} \big( -2 \tth  y_2 (\tth) + o\big(\tth  y_2 (\tth) \big) \big) \tth e^{\frac{\tth}{4}} 
\stackrel{\eqref{y2}}{=} 
%\lim_{\tth\to+\infty}  -\frac34 \tth^2 e^{-\frac{\tth}{4}}  + %o\big(\tth^2 e^{-\frac{\tth}{4}} \big) = 
0\, .  
$$
Thus, by substituting $y_2(\tth) = \frac38 e^{-\frac{\tth}{2}} \big(1+z_2(\tth) \big) $ where $z_2(\tth) = o(1) $ as $\tth\to +\infty$,  into \eqref{p=2as} we obtain
%$$
%\frac12 e^{-\frac{\tth}{2}} \big(1+z_2(\tth) \big)  = 
%\sqrt{  \frac14 + \frac38 e^{-\frac{\tth}{2}} \big(1+z_2(\tth) \big)  } %\, 
%e^{ -   \frac{\tth}{2}  }  \big(1+o\big(\tth^{-1}e^{-\frac{\tth}{4}}\big) \big)
%  + O( e^{-  \tth}) \, ,
%$$
%namely
$$
\frac12 \big(1+z_2(\tth) \big)  = 
\frac12 \Big( 1 + \frac34 e^{-\frac{\tth}{2}} \big(1+z_2(\tth) \big)  \Big)  \, 
 \big(1+o\big(\tth^{-1}e^{-\frac{\tth}{4}}\big) \big)
  + O\big(e^{ -   \frac{\tth}{2}  } \big)  = \frac12 
 \big(1+o\big(\tth^{-1}e^{-\frac{\tth}{4}}\big) \big)\,,
$$
which gives $z_2(\tth) = o\big(\tth^{-1}e^{-\frac{\tth}{4}}\big) $.
\end{proof}
By \eqref{expuphi} and since 
\begin{equation}\label{a+b}
\tanh(a+b) = \dfrac{\tanh(a)+\tanh(b)}{1+\tanh(a)\tanh(b)} \, ,
\quad \forall a, b \in \bR \, , 
\end{equation}
we get 
\begin{subequations}\label{tanhall}
\begin{align}\notag
\tanh\big(\tth \uphi(2,\tth) \big) &\;\,=\;\, \tanh\Big( \frac{\tth}4 + \frac38 \tth e^{-\frac{\tth}{2}}  \big(1+o\big(\tth^{-1}e^{-\frac{\tth}{4}}\big) \big)    \Big) = \dfrac{\tanh\big( \frac{\tth}4\big)+\tanh\Big(\frac38 \tth e^{-\frac{\tth}{2}}  \big(1+o\big(\tth^{-1}e^{-\frac{\tth}{4}}\big)\big)    \Big)}{1+\tanh\big( \frac{\tth}4\big)\tanh\Big(\frac38 \tth e^{-\frac{\tth}{2}}  \big( 1+o\big(\tth^{-1}e^{-\frac{\tth}{4}}\big) \big)    \Big) }\\
&\stackrel{\eqref{tanhasymp}}{=} \dfrac{1+\frac38 \tth e^{-\frac{\tth}{2}} -2e^{-\frac{\tth}{2}} +o\big( e^{-\frac{3}{4}\tth} \big) }{1+\frac38 \tth e^{-\frac{\tth}{2}} +o\big( e^{-\frac{3}{4}\tth}\big)} = 1- 2e^{-\frac{\tth}{2}} + O\big(e^{-\frac{3}{4}\tth} \big)\, . 
\end{align}
Similarly we deduce that
\begin{equation}
\tanh\Big(\tth\big(1+\uphi(2,\tth)\big)\Big) =  1 + O\big(e^{-\frac{3}{4}\tth} \big)\, ,\qquad \tanh\Big(\tth\big(2+\uphi(2,\tth)\big)\Big) =  1 + O\big(e^{-\frac{3}{4}\tth} \big)\, . 
\end{equation}
\end{subequations}
By \eqref{tanhall} and \eqref{expuphi} the terms in \eqref{tj2} satisfy, as $\tth \to +\infty $,
 \begin{equation} \label{tj2asymp}
\begin{aligned}
&\Omega_0^{(2)} = \sqrt{\Big(\frac14 + \frac38 e^{-\frac{\tth}{2}} + O\big(e^{-\frac{3}{4}\tth} \big) \Big)\Big(1-2 e^{-\frac{\tth}{2}} + O\big(e^{-\frac{3}{4}\tth} \big)\Big)} = \frac12 - \frac18 e^{-\frac{\tth}{2}} + O\big(e^{-\frac{3}{4}\tth} \big)
  \, , \\ 
  &t_0^{(2)}= \sqrt{\frac{\frac14 + \frac38 e^{-\frac{\tth}{2}} + O\big(e^{-\frac{3}{4}\tth} \big) }{1-2 e^{-\frac{\tth}{2}} + O\big(e^{-\frac{3}{4}\tth} \big)}} = \frac12 + \frac78 e^{-\frac{\tth}{2}} + O\big(e^{-\frac{3}{4}\tth} \big)\, , \\ 
  &\Omega_1^{(2)} = \sqrt{\Big(\frac54 + \frac38 e^{-\frac{\tth}{2}} + O\big(e^{-\frac{3}{4}\tth} \big) \Big)\Big(1 + O\big(e^{-\frac{3}{4}\tth} \big)\Big)} = \frac{\sqrt{5}}{2}+ \frac{3 }{8 \sqrt{5}}e^{-\frac{\tth}{2}}  + O\big(e^{-\frac{3}{4}\tth} \big)
  \, , \\ 
    &\Omega_2^{(2)} =  \sqrt{\Big(\frac94 + \frac38 e^{-\frac{\tth}{2}} + O\big(e^{-\frac{3}{4}\tth} \big) \Big)\Big(1 + O\big(e^{-\frac{3}{4}\tth} \big)\Big)} = \frac32 + \frac18 e^{-\frac{\tth}{2}} + O\big(e^{-\frac{3}{4}\tth} \big)
  \, , \\ 
  &t_1^{(2)} = \frac{\sqrt{5}}{2}+\frac{3 }{8 \sqrt{5}} e^{-\frac{\tth}{2}}+ O\big(e^{-\frac{3}{4}\tth} \big)\, , \qquad 
  t_2^{(2)}=  \frac32 + \frac18 e^{-\frac{\tth}{2}} + O\big(e^{-\frac{3}{4}\tth} \big)
  \, .
  \end{aligned}
  \end{equation}
By \eqref{dwch}, \eqref{dwap12} and \eqref{tj2asymp}, 
 the term in \eqref{primob1} satisfy, as $\tth\to+\infty$,
\begin{equation}\label{ultimaformula}
b_0^{(2)} = -\frac{3\sqrt{3}}{16}e^{-\frac{\tth}{2}}+ O\big(e^{-\frac{3}{4}\tth} \big)\, .
\end{equation}
Note that $ b_0^{(2)} $  vanishes in the deep-water limit as stated in \cite[(6.36)]{BCMV}. Similarly the terms in \eqref{secondob1}-\eqref{terzob1} are, as $\tth\to+\infty$,
\begin{equation}\label{ultimaformulabis}
\begin{aligned}
\betone{1}{1}{-} &= \frac{\sqrt{15}}{16}+\frac{143 \sqrt{5}-225}{640 \sqrt{3}}e^{-\frac{\tth}{2}}+ O\big(e^{-\frac{3}{4}\tth} \big) \, , \\
\betone{1}{1}{+}  &= \frac{\sqrt{15}}{16}+\frac{143 \sqrt{5}+225 }{640 \sqrt{3}}e^{-\frac{\tth}{2}}+ O\big(e^{-\frac{3}{4}\tth} \big)\,,
\end{aligned}
\end{equation}
whose difference is 
\begin{equation}\label{b1p=2}
    \betone{1}{1}{} := \betone{1}{1}{+}- \betone{1}{1}{-} = \frac{15 \sqrt{3}}{64}e^{-\frac{\tth}{2}} + O\big(e^{-\frac{3}{4}\tth} \big) \, .
\end{equation}
Note that $ \betone{1}{1}{} $ vanishes in the deep-water limit as stated in \cite[Lemma 6.7]{BCMV}.

By \eqref{ultimaformula}, \eqref{ultimaformulabis}, \eqref{b1p=2} and \eqref{splittingb1p2} we deduce the expansion \eqref{expbeta1p=2} for $\beta_1^{(2)}$.

\subsection{Case $ \mathtt{p} =3$ }
We now prove  \eqref{expbeta1p=3}. The computations become more involved. 
According to \cite[(5.7)-(5.8)]{BCMV}, the coefficient $\beta_1^{(3)} (\tth) $  is given  by
\begin{equation}
 \begin{aligned}
\label{splittingb1p3}
\beta_1^{(3)}(\tth)  &=b_0^{(3)}-\betonet{1}{1}{-}+\betonet{1}{1}{+}-\betonet{1}{2}{-}+\betonet{1}{2}{+} \\ 
&\ \ -\betonet{2}{1,2}{-,+}+\betonet{2}{1,2}{-,-}-\betonet{2}{1,2}{+,-}+\betonet{2}{1,2}{+,+}
\end{aligned}
\end{equation}
(in this case there are $9 $ terms) where 
\begin{subequations}\label{b11b12b13p=3}
\begin{align} \label{zerob1p=3}
b_0^{(3)} & :=\frac14 \sqrt{\Omega_0^{(3)}\Omega_{3}^{(3)}} \big(a_3^{[3]}   + p_3^{[3]}(t_0^{(3)}-t_3^{(3)}) \big)\, , \\  
 \label{primob1p=3} 
\betonet{1}{1}{-}  & := \frac{\Omega_{1}^{(3)}\sqrt{\Omega_0^{(3)}\Omega_{3}^{(3)}}}{16(\ch+\Omega_{1}^{(3)}-\Omega_0^{(3)})}  \big(a_1^{[1]} +p_1^{[1]}( t_{0}^{(3)} + t_{1}^{(3)})  \big) \big(a_2^{[2]}+p_2^{[2]}(t_1^{(3)}-t_3^{(3)}) \big)\, , \\
 \label{secondob1p=3}
\betonet{1}{1}{+}  & := \frac{\Omega_{1}^{(3)}\sqrt{\Omega_0^{(3)}\Omega_{3}^{(3)}}}{16(\ch-\Omega_{1}^{(3)}-\Omega_0^{(3)})}  \big(a_1^{[1]} +p_1^{[1]}( t_{0}^{(3)} - t_{1}^{(3)})  \big) \big(a_2^{[2]}-p_2^{[2]}(t_1^{(3)}+t_3^{(3)}) \big)\, , \\ 
 \label{terzob1p=3} 
\betonet{1}{2}{-}  & := \frac{\Omega_{2}^{(3)}\sqrt{\Omega_0^{(3)}\Omega_{3}^{(3)}}}{16(2\ch+\Omega_{2}^{(3)}-\Omega_0^{(3)})}  \big(a_1^{[1]} +p_1^{[1]}( t_{2}^{(3)} - t_{3}^{(3)})  \big) \big(a_2^{[2]}+p_2^{[2]}(t_0^{(3)}+t_2^{(3)}) \big)\, , \\
 \label{quartob1p=3} 
\betonet{1}{2}{+}  & := \frac{\Omega_{2}^{(3)}\sqrt{\Omega_0^{(3)}\Omega_{3}^{(3)}}}{16(2\ch-\Omega_{2}^{(3)}-\Omega_0^{(3)})}  \big(a_1^{[1]} -p_1^{[1]}( t_{2}^{(3)} + t_{3}^{(3)})  \big) \big(a_2^{[2]}+p_2^{[2]}(t_0^{(3)}-t_2^{(3)}) \big)\, , \\
 \label{quintob1p=3}
\betonet{2}{1,2}{-,+}  & :=  \frac{\Omega_{1}^{(3)}\Omega_2^{(3)}\sqrt{\Omega_0^{(3)}\Omega_{3}^{(3)}} \big(a_1^{[1]} +p_1^{[1]}( t_{0}^{(3)} +t_{1}^{(3)} )  \big) \big(a_1^{[1]} +p_1^{[1]}( t_{1}^{(3)} -t_{2}^{(3)} )  \big)\big(a_1^{[1]} -p_1^{[1]}( t_{2}^{(3)}+t_{3}^{(3)} )  \big)}{64(\ch+\Omega_{1}^{(3)}-\Omega_0^{(3)})(2\ch-\Omega_{2}^{(3)}-\Omega_0^{(3)})}   \, , \\
 \label{sestob1p=3} 
\betonet{2}{1,2}{-,-}  & := \frac{\Omega_{1}^{(3)}\Omega_2^{(3)} \sqrt{\Omega_0^{(3)}\Omega_{3}^{(3)}}\big(a_1^{[1]} +p_1^{[1]}( t_{0}^{(3)} +t_{1}^{(3)} )  \big) \big(a_1^{[1]} +p_1^{[1]}( t_{1}^{(3)} +t_{2}^{(3)} )  \big)\big(a_1^{[1]} +p_1^{[1]}( t_{2}^{(3)} -t_{3}^{(3)} )  \big) }{64(\ch+\Omega_{1}^{(3)}-\Omega_0^{(3)})(2\ch+\Omega_{2}^{(3)}-\Omega_0^{(3)})}   \, ,\\
 \label{settimob1p=3}
\betonet{2}{1,2}{+,-}  & :=  \frac{\Omega_{1}^{(3)}\Omega_2^{(3)} \sqrt{\Omega_0^{(3)}\Omega_{3}^{(3)}} \big(a_1^{[1]} +p_1^{[1]}( t_{0}^{(3)} -t_{1}^{(3)} )  \big) \big(a_1^{[1]} +p_1^{[1]}(- t_{1}^{(3)}+ t_{2}^{(3)} )  \big)\big(a_1^{[1]} +p_1^{[1]}( t_{2}^{(3)} -t_{3}^{(3)} )  \big)}{64(\ch-\Omega_{1}^{(3)}-\Omega_0^{(3)})(2\ch+\Omega_{2}^{(3)}-\Omega_0^{(3)})} \! , \\
 \label{ottavob1p=3} 
\betonet{2}{1,2}{+,+} & := \frac{\Omega_{1}^{(3)}\Omega_2^{(3)} \sqrt{\Omega_0^{(3)}\Omega_{3}^{(3)}} \big(a_1^{[1]} +p_1^{[1]}( t_{0}^{(3)} -t_{1}^{(3)} )  \big) \big(a_1^{[1]} -p_1^{[1]}( t_{1}^{(3)} +t_{2}^{(3)} )  \big)\big(a_1^{[1]} -p_1^{[1]}( t_{2}^{(3)} +t_{3}^{(3)} )  \big) }{64(\ch-\Omega_{1}^{(3)}-\Omega_0^{(3)})(2\ch-\Omega_{2}^{(3)}-\Omega_0^{(3)})}  \, ,
\end{align}
\end{subequations}
with, in view \cite[(5.6)]{BCMV}, for $ j = 0,1,2,3 $,  
\begin{equation} \label{tj3}
  \Omega_j^{(3)} := \sqrt{\big(j+\uphi(3,\tth)\big)\tanh\Big(\tth\big(j+\uphi(3,\tth)\big)\Big)}
  \, ,\qquad 
  t_j^{(3)} = \sqrt{\frac{j+\uphi(3,\tth)}{\tanh\Big(\tth\big(j+\uphi(3,\tth)\big)\Big)}}\,  ,
  \end{equation}
%\begin{equation} \label{tj3}
%\begin{alignedat}{2}
%&\Omega_0^{(3)} = \sqrt{\uphi(3,\tth)\tanh(\tth\uphi(3,\tth))}
 % \, ,\qquad &&t_0^{(3)}= \sqrt{\frac{\uphi(3,\tth)}{\tanh(\tth\uphi(3,\tth))}}\, , \\ 
  %&\Omega_1^{(3)} = \sqrt{\big(1+\uphi(3,\tth)\big)\tanh\Big(\tth\big(1+\uphi(3,\tth)\big)\Big)}
 % \, ,\quad  &&t_1^{(3)} = \sqrt{\frac{1+\uphi(3,\tth)}{\tanh\Big(\tth\big(1+\uphi(3,\tth)\big)\Big)}}\, , \\ 
 % &\Omega_2^{(3)} = \sqrt{\big(2+\uphi(3,\tth)\big)\tanh\Big(\tth\big(2+\uphi(3,\tth)\big)\Big)}
 % \, ,\quad &&t_2^{(3)}= \sqrt{\frac{2+\uphi(3,\tth)}{\tanh\Big(\tth\big(2+\uphi(3,\tth)\big)\Big)}}\, %, \\
  %  &\Omega_3^{(3)} = \sqrt{\big(3+\uphi(3,\tth)\big)\tanh\Big(\tth\big(3+\uphi(3,\tth)\big)\Big)}
  %\, ,\quad &&t_3^{(3)}= \sqrt{\frac{3+\uphi(3,\tth)}{\tanh\Big(\tth\big(3+\uphi(3,\tth)\big)\Big)}}\, ,
  %\end{alignedat}
 % \end{equation}
with the  wavenumber $\uphi(3,\tth)$  defined in  \cite[Lemma 2.7]{BCMV},
and the $a_\ell^{[\ell]}$, $p_\ell^{[\ell]} $ 
are defined in \eqref{ap12} for $\ell=1,2$ and, for $ \ell = 3 $, are given by, cfr. \cite[(A.59)-(A.60)]{BMV_ed},
\begin{equation}
\label{ap3}
p_3^{[3]} =  -\frac{\ch^{12}+17 \ch^8+51 \ch^4+27}{32\ch^{13}}\, ,\quad
a_3^{[3]} =  \frac{-\ch^{16}-98 \ch^{12}+252 \ch^8-318 \ch^4-27}{64\ch^{14}}\, .
\end{equation}
Using \eqref{dwch},  
the terms in \eqref{ap3} admit the
expansion, as $\tth\to +\infty$, 
\begin{equation}\label{dwap3}
    p_3^{[3]}=-3 - 28 e^{-2\tth} +O\big(e^{-4\tth} \big)  \,, \quad a_3^{[3]}=-3 - 35e^{-2\tth} +O\big(e^{-4\tth} \big)\, . 
\end{equation}
For the terms in \eqref{tj3} we first need a slight improvement of formula \cite[(2.31)]{BCMV}.
\begin{lem} The wavenumber $\uphi(3,\tth)$ in {\em \cite[Lemma 2.7]{BCMV}} satisfies 
\begin{equation}\label{expuphip=3}
\uphi(3,\tth) = 1 -\frac83 e^{-2\tth}  \big( 1 +o\big(\tth^{-1}e^{-\tth}\big) \big) 
\quad \text{as} \quad \tth \to +\infty \, .
\end{equation}
\end{lem}
\begin{proof}
   Let us recall from \cite[Appendix A]{BCMV} that  
   \begin{align}\label{y3h}
   \uphi(3,\tth) = 1 + y_3(\tth) \qquad \text{where} \qquad 
   y_3(\tth) = -\frac83 e^{-2\tth} \big(1+o(1) \big)
   \end{align}
    satisfies, cfr.  formula \cite[(A.11)]{BCMV} and the one below \cite[(A.12)]{BCMV},
\begin{equation}\label{p=2as2}
\frac34 y_3 (\tth) = - 3  e^{ - 2 \tth } + \sqrt{1 + y_3 (\tth) } 
e^{- 2 \tth} e^{- 2 \tth y_3 (\tth)}
+ O(y_3^2 (\tth) )   + O(e^{- 4 \tth }) \ .
\end{equation}
We observe that $e^{ - 2 \tth  y_3 (\tth) }  = 1+o\big(\tth^{-1}e^{-\tth}\big) $, since
$$
\lim_{\tth\to+\infty}\big( e^{ - 2 \tth  y_3 (\tth) } - 1\big) \tth e^{\tth}=\lim_{\tth\to+\infty} \big( -2 \tth  y_3 (\tth) + o\big(\tth  y_3 (\tth) \big) \big) \tth e^{\tth} \stackrel{\eqref{y3h}}{=} 0\, .  
$$
Thus, by substituting $y_3(\tth) = -\frac83 e^{-2{\tth}} \big(1+z_3(\tth) \big) $ where $z_3(\tth) = o(1) $ as $\tth\to +\infty$,  into \eqref{p=2as2} we obtain
%$$
%-2 e^{-2\tth} \big(1+z_3(\tth) \big)  = -3 e^{-2{\tth}} + 
%\sqrt{ 1 -\frac83 e^{-2{\tth}} \big(1+z_3(\tth) \big)  } \, 
%e^{ -  2\tth } \big( 1+o\big(\tth^{-1}e^{-\tth}\big)\big)
%  + O( e^{-  4\tth}) \, ,
%$$
%namely
$$
-2\big(1+z_3(\tth) \big)  = 
-3 +  \Big( 1 - \tfrac43 e^{-2{\tth}} \big(1+z_3(\tth) \big)  \Big)  \, 
 \big(1+o\big(\tth^{-N}\big) \big)
  + O\big(e^{-3\tth} \big)  = 
 -2 \big(1+o\big(\tth^{-1}e^{-\tth}\big) \big)\,,
$$
which gives $z_3(\tth) = o\big(\tth^{-1}e^{-\tth}\big) $.
\end{proof}
By \eqref{expuphip=3} and \eqref{a+b} we get 
\begin{subequations}\label{tanhallp=3}
\begin{align}\notag
\tanh\big(\tth \uphi(3,\tth) \big) &\;\,=\;\, \tanh\Big( \tth - \frac83 \tth e^{-2\tth}  \big( 1 +o\big(\tth^{-1}e^{-\tth}\big)\big)    \Big) = \dfrac{\tanh(\tth)-\tanh\Big(\frac83 \tth e^{-2\tth}  \big( 1 +o\big(\tth^{-1}e^{-\tth}\big)\big)    \Big)}{1-\tanh(\tth)\tanh\Big(\frac83 \tth e^{-2\tth}  \big( 1 +o\big(\tth^{-1}e^{-\tth}\big)\big)    \Big)  }\\
&\stackrel{\eqref{tanhasymp}}{=} \dfrac{1-\frac83 \tth e^{-2\tth} -2 e^{-2\tth}+o\big(  e^{-3\tth}\big) }{1-\frac83 \tth e^{-2\tth}+o\big(e^{-3\tth}\big)} = 1- 2e^{-2\tth} + O\big(e^{-3\tth} \big)\, .
\end{align}
% where in the last step we weaken the little-$o$ into a big-$O$.
Similarly we deduce that 
\begin{equation}
%\begin{aligned}
\tanh\Big(\tth\big(j +\uphi(3,\tth)\big)\Big) =  1 + O\big(e^{-3\tth} \big)\, ,\    \quad
j = 2,3,4 \, . 
%\\ &\tanh\Big(\tth\big(2+\uphi(3,\tth)\big)\Big) = % 1 + O\big(e^{-3\tth} \big)\, , 
%\end{aligned}\quad \tanh\Big(\tth\big(3+\uphi(3,\tth)\big)\Big) =  1 + O\big(e^{-3\tth} \big)\, .
\end{equation}
\end{subequations}
By \eqref{tanhallp=3} and \eqref{expuphip=3}, the terms in \eqref{tj3} satisfy, as $\tth \to +\infty $,
 \begin{equation} \label{tj3asymp}
\begin{alignedat}{2}
&\Omega_0^{(3)} 
= 1 - \tfrac73 e^{-2{\tth}} + O\big(e^{-3{\tth}} \big)
  \, , \qquad &&t_0^{(3)}= \sqrt{\frac{\uphi(3,\tth)}{\tanh(\tth\uphi(3,\tth))}}  
  = 1 - \tfrac13 e^{-2\tth} + O\big(e^{-3\tth} \big)\, , \\ 
  &\Omega_1^{(3)}  
  = \sqrt{2} - \tfrac{2\sqrt2}{3}e^{-2\tth} + O\big(e^{-3\tth} \big)
  \, ,\quad &&t_1^{(3)}   = \sqrt{2} - \tfrac{2\sqrt2}{3}e^{-2\tth} + O\big(e^{-3\tth} \big)\,, \\ 
    &\Omega_2^{(3)}
    = \sqrt{3} - \tfrac{4}{3\sqrt3} e^{-2{\tth}} + O\big(e^{-3{\tth}} \big)
  \, ,\qquad  &&t_2^{(3)}=  \sqrt{3} - \tfrac{4}{3\sqrt3} e^{-2{\tth}} + O\big(e^{-3{\tth}} \big)\, ,  \\ 
      &\Omega_3^{(3)} 
    = 2 - \tfrac{2}{3} e^{-2{\tth}} + O\big(e^{-3{\tth}} \big)
  \, ,\quad
  &&t_3^{(3)}  = 2 - \tfrac{2}{3} e^{-2{\tth}} + O\big(e^{-3{\tth}} \big) \,\ .
  \end{alignedat}
  \end{equation}
We are now in a position to compute the asymptotic expansion of the summands in \eqref{b11b12b13p=3}. We group the latter terms depending on the number of visited intermediate harmonics.\smallskip

\noindent{\bf No intermediate harmonics}. By \eqref{dwap3} and \eqref{tj3asymp}, the term in \eqref{zerob1p=3} satisfies  
\begin{equation}\label{b0p=3}
    b_0^{(3)}= -2 \sqrt{2} e^{-2\tth} + O\big(e^{-3\tth} \big) 
    \quad \text{as} \quad \tth\to +\infty   \, .
\end{equation}
Note that $ b_0^{(3)} $ vanishes in the deep water limit, as stated in \cite[(6.36)]{BCMV}.\smallskip

\noindent{\bf One intermediate harmonic}. We consider the terms in \eqref{primob1p=3} -\eqref{quartob1p=3}  as $\tth\to +\infty$. By \eqref{dwch}, \eqref{dwap12}  and \eqref{tj3asymp} we obtain
\begin{equation}
\begin{aligned}
    \betonet{1}{1}{-} &= \frac{1}{2} +\frac{5}{6} \left(5+2 \sqrt{2}\right)e^{-2\tth} + O\big(e^{-3\tth} \big)\, ,\\
    \betonet{1}{1}{+} &=\frac{1}{2}+ \frac{5}{6} \left(5-2 \sqrt{2}\right) e^{-2\tth} + O\big(e^{-3\tth} \big)\, .
\end{aligned}
\end{equation}
Note that 
\begin{equation}\label{b1p=3}
    \betonet{1}{1}{} := \betonet{1}{1}{+}- \betonet{1}{1}{-} =-\tfrac{1}{3} 10 \sqrt{2}  e^{-2\tth} + O\big(e^{-3\tth} \big)
\end{equation}
vanishes in the deep-water limit as stated in \cite[Lemma 6.7]{BCMV}. Similarly
\begin{equation}
\begin{aligned}
    \betonet{1}{2}{-} &=\frac{\sqrt 3}{2\sqrt 2}+\left(\frac{125}{6 \sqrt{6}}-4 \sqrt{2}\right)e^{-2\tth} + O\big(e^{-3\tth} \big)\, ,\\
    \betonet{1}{2}{+} &=  \frac{\sqrt 3}{2\sqrt 2}+\left(\frac{125}{6 \sqrt{6}}+4 \sqrt{2}\right)e^{-2\tth} + O\big(e^{-3\tth} \big)\, ,
\end{aligned}
\end{equation}
whose difference 
\begin{equation}\label{b2p=3}
    \betonet{1}{2}{} := \betonet{1}{2}{+}- \betonet{1}{2}{-} = 8 \sqrt{2}
     e^{-2\tth} + O\big(e^{-3\tth} \big)
\end{equation}
vanishes in the deep-water limit as stated in \cite[Lemma 6.7]{BCMV}. \smallskip

\noindent{\bf Two intermediate harmonics}. We consider the terms in \eqref{quintob1p=3}-\eqref{ottavob1p=3} as $\tth\to +\infty$. By \eqref{dwch}, \eqref{dwap12}  and \eqref{tj3asymp} we obtain
\begin{equation}
\begin{aligned}
    \betonet{2}{1,2}{-,+} &= \frac{1}{4} \left(-3 \sqrt{3}-\sqrt{6}-3\right)+\frac{1}{36} \left(18 \sqrt{2}+6 \sqrt{3}+31 \sqrt{6}-39\right) e^{-2\tth} + O\big(e^{-3\tth} \big) \, ,\\
    \betonet{2}{1,2}{-,-} &= \frac{1}{4} \left(-3 \sqrt{3}-\sqrt{6}+3\right)+\frac{1}{36} \left(-18 \sqrt{2}+6 \sqrt{3}+31 \sqrt{6}+39\right) e^{-2\tth} + O\big(e^{-3\tth} \big)\, ,\\
    \betonet{2}{1,2}{+,-} &=\frac{1}{4} \left(-3 \sqrt{3}+\sqrt{6}+3\right) +\frac{1}{36} \left(18 \sqrt{2}+6 \sqrt{3}-31 \sqrt{6}+39\right) e^{-2\tth} + O\big(e^{-3\tth} \big)\, ,\\
     \betonet{2}{1,2}{+,+} &=\frac{1}{4} \left(-3 \sqrt{3}+\sqrt{6}-3\right)+\frac{1}{36} \left(-18 \sqrt{2}+6 \sqrt{3}-31 \sqrt{6}-39\right) e^{-2\tth} + O\big(e^{-3\tth} \big)\, ,
\end{aligned}
\end{equation}
whose signed sum 
\begin{equation}\label{b12p=3}
    \betonet{2}{1,2}{} :=  \betonet{2}{1,2}{+,+}-\betonet{2}{1,2}{-,+}+\betonet{2}{1,2}{-,-}-\betonet{2}{1,2}{+,-}  =-2 \sqrt{2}e^{-2\tth} + O\big(e^{-3\tth} \big)
\end{equation}
vanishes in the deep-water limit as stated in \cite[Lemma 6.7]{BCMV}.

The sum of the terms in \eqref{b0p=3}, \eqref{b1p=3}, \eqref{b2p=3} and \eqref{b12p=3} gives the asymptotic expansion \eqref{expbeta1p=3}. 

\subsection{Case $ \mathtt{p} =4$}\label{sec:4}

In this  section we prove  \eqref{expbeta1p=4}. 
The computations are  much more involved and we make use of Mathematica. 
In this case there are 27 terms. 
According to \cite[(5.7)-(5.8)]{BCMV}, the coefficient $\beta_1^{(4)} (\tth) $  is 
\begin{align} \label{beta1p4}
 \footnotesize \beta_1^{(4)} (\tth) 
 &=b_0^{[4]}-\betoneq{1}{1}{-}+\betoneq{1}{1}{+}-\betoneq{1}{2}{-}+\betoneq{1}{2}{+}-\betoneq{1}{3}{-}+\betoneq{1}{3}{+} \\ \notag
&\ \ -\betoneq{2}{1,2}{-,+} +\betoneq{2}{1,2}{-,-}-\betoneq{2}{1,2}{+,-}+\betoneq{2}{1,2}{+,+} \\  \notag
&\ \ -\betoneq{2}{1,3}{-,+} +\betoneq{2}{1,3}{-,-}-\betoneq{2}{1,3}{+,-}+\betoneq{2}{1,3}{+,+} \\  \notag
&\ \ -\betoneq{2}{2,3}{-,+} +\betoneq{2}{2,3}{-,-}-\betoneq{2}{2,3}{+,-}+\betoneq{2}{2,3}{+,+} \\ \notag
&\ \ -\betoneq{3}{1,2,3}{-,+,+} +\betoneq{3}{1,2,3}{-,-,+}-\betoneq{3}{1,2,3}{+,-,+}+\betoneq{3}{1,2,3}{+,+,+} \\ \notag
&\ \ -\betoneq{3}{1,2,3}{+,+,-}+\betoneq{3}{1,2,3}{+,-,-} -\betoneq{3}{1,2,3}{-,-,-}+\betoneq{3}{1,2,3}{-,+,-}
\end{align}
where, denoting for simplicity $\Omega_j:=\Omega_j^{(4)} $ and $t_j=t_j^{(4)} $ for $j=1,2,3,4$,
\begin{subequations}\label{b1partsp4}
\begin{align}\notag %\label{b10p4} 
b_0^{[4]} & :=\frac14 \sqrt{\Omega_{0}\Omega_{4}} \big(a_4^{[4]}   + p_4^{[4]}(t_0-t_4) \big)\, , \\  
 \label{b113p4} 
\betoneq{1}{1}{-}  & := \frac{\Omega_{1}\sqrt{\Omega_{0} \Omega_{4}}}{16(\ch+\Omega_{1}-\Omega_{0})}  \big(a_1^{[1]} +p_1^{[1]}( t_{0} + t_{1})  \big) \big(a_3^{[3]}+p_3^{[3]}(t_{1}-t_{4}) \big)\, , \\
 \label{b113p4bis}
\betoneq{1}{1}{+}  & := \frac{\Omega_{1}\sqrt{\Omega_{0}\Omega_{4}}}{16(\ch-\Omega_{1}-\Omega_{0})}  \big(a_1^{[1]} +p_1^{[1]}( t_{0} - t_{1})  \big) \big(a_3^{[3]}-p_3^{[3]}(t_1+t_4) \big)\, , \\ 
 \label{b114p4} 
\betoneq{1}{2}{-}  & := \frac{\Omega_{2}\sqrt{\Omega_{0}\Omega_{4}}}{16(2\ch+\Omega_{2}-\Omega_{0})}  \big(a_2^{[2]} +p_2^{[2]}( t_{2} - t_{4})  \big) \big(a_2^{[2]}+p_2^{[2]}(t_0+t_2) \big)\, , \\
 \label{b114p4bis} 
\betoneq{1}{2}{+}  & := \frac{\Omega_{2}\sqrt{\Omega_{0}\Omega_{4}}}{16(2\ch-\Omega_{2}-\Omega_0)}  \big(a_2^{[2]} -p_2^{[2]}( t_{2} + t_{4})  \big) \big(a_2^{[2]}+p_2^{[2]}(t_0-t_2) \big)\, , \\
 \label{b115p4} 
\betoneq{1}{3}{-}  & := \frac{\Omega_{3}\sqrt{\Omega_0\Omega_{4}}}{16(3\ch+\Omega_{3}-\Omega_0)}  \big(a_1^{[1]} +p_1^{[1]}( t_{3} - t_{4})  \big) \big(a_3^{[3]}+p_3^{[3]}(t_0+t_3) \big)\, , \\
 \label{b115p4bis} 
\betoneq{1}{3}{+}  & := \frac{\Omega_{3}\sqrt{\Omega_0\Omega_{4}}}{16(3\ch-\Omega_{3}-\Omega_0)}  \big(a_1^{[1]} -p_1^{[1]}( t_{3} + t_{4})  \big) \big(a_3^{[3]}+p_3^{[3]}(t_0-t_3) \big)\, ,
 \\
 \label{b1234p4}
\betoneq{2}{1,2}{-,+}  & :=  \frac{\Omega_{1}\Omega_2 \sqrt{\Omega_0\Omega_{4}} \big(a_1^{[1]} +p_1^{[1]}( t_{0} +t_{1} )  \big) \big(a_1^{[1]} +p_1^{[1]}( t_{1} -t_{2} )  \big)\big(a_2^{[2]} -p_2^{[2]}( t_{2} +t_{4} )  \big)}{64(\ch+\Omega_{1}-\Omega_0)(2\ch-\Omega_{2}-\Omega_0)}   \, , \\
 \label{b1234p4bis} 
\betoneq{2}{1,2}{-,-}  & := \frac{\Omega_{1}\Omega_2 \sqrt{\Omega_0\Omega_{4}}\big(a_1^{[1]} +p_1^{[1]}( t_{0} +t_{1} )  \big) \big(a_1^{[1]} +p_1^{[1]}( t_{1} +t_{2} )  \big)\big(a_2^{[2]} +p_2^{[2]}( t_{2} -t_{4} )  \big) }{64(\ch+\Omega_{1}-\Omega_0)(2\ch+\Omega_{2}-\Omega_0)}   \, ,\\
 \label{b1234p4tris}
\betoneq{2}{1,2}{+,-}  & :=  \frac{\Omega_{1}\Omega_{2} \sqrt{\Omega_0\Omega_{4}} \big(a_1^{[1]} +p_1^{[1]}( t_{0} -t_{1} )  \big) \big(a_1^{[1]} +p_1^{[1]}(- t_{1}+ t_{2} )  \big)\big(a_2^{[2]} +p_2^{[2]}( t_{2} -t_{4} )  \big)}{64(\ch-\Omega_{1}-\Omega_0)(2\ch+\Omega_{2}-\Omega_0)} \, , \\
 \label{b1234p4quat} 
\betoneq{2}{1,2}{+,+} & := \frac{\Omega_{1}\Omega_{2} \sqrt{\Omega_0\Omega_{4}} \big(a_1^{[1]} +p_1^{[1]}( t_{0} -t_{1} )  \big) \big(a_1^{[1]} -p_1^{[1]}( t_{1} +t_{2} )  \big)\big(a_2^{[2]} -p_2^{[2]}( t_{2} +t_{4} )  \big) }{64(\ch-\Omega_{1}-\Omega_0)(2\ch-\Omega_{2}-\Omega_0)}  \, ,
 \\
 \label{b1235p4}
\betoneq{2}{1,3}{-,+}  & :=  \frac{\Omega_{1}\Omega_{3} \sqrt{\Omega_0\Omega_{4}} \big(a_1^{[1]} +p_1^{[1]}( t_{0} +t_{1} )  \big) \big(a_2^{[2]} +p_2^{[2]}( t_{1} -t_{3} )  \big)\big(a_1^{[1]} -p_1^{[1]}( t_{3} +t_{4} )  \big)}{64(\ch+\Omega_{1}-\Omega_0)(3\ch-\Omega_{3}-\Omega_0)}   \, , \\
 \label{b1235p4bis} 
\betoneq{2}{1,3}{-,-}  & := \frac{\Omega_{1}\Omega_{3} \sqrt{\Omega_0\Omega_{4}}\big(a_1^{[1]} +p_1^{[1]}( t_{0} +t_{1} )  \big) \big(a_2^{[2]} +p_2^{[2]}( t_{1} +t_{3} )  \big)\big(a_1^{[1]} +p_1^{[1]}( t_{3} -t_{4} )  \big) }{64(\ch+\Omega_{1}-\Omega_0)(3\ch+\Omega_{3}-\Omega_0)}   \, ,\\
 \label{b1235p4tris}
\betoneq{2}{1,3}{+,-}  & :=  \frac{\Omega_{1}\Omega_{3} \sqrt{\Omega_0\Omega_{4}} \big(a_1^{[1]} +p_1^{[1]}( t_{0} -t_{1} )  \big) \big(a_2^{[2]} +p_2^{[2]}(- t_{1}+ t_{3} )  \big)\big(a_1^{[1]} +p_1^{[1]}( t_{3} -t_{4} )  \big)}{64(\ch-\Omega_{1}-\Omega_0)(3\ch+\Omega_{3}-\Omega_0)} \, , \\
 \label{b1235p4quat} 
\betoneq{2}{1,3}{+,+} & := \frac{\Omega_{1}\Omega_{3} \sqrt{\Omega_0\Omega_{4}} \big(a_1^{[1]} +p_1^{[1]}( t_{0} -t_{1} )  \big) \big(a_2^{[2]} -p_2^{[2]}( t_{1} +t_{3} )  \big)\big(a_1^{[1]} -p_1^{[1]}( t_{3} +t_{4} )  \big) }{64(\ch-\Omega_{1}-\Omega_0)(3\ch-\Omega_{3}-\Omega_0)}  \, ,
 \\
 \label{b1245p4}
\betoneq{2}{2,3}{-,+}  & :=  \frac{\Omega_{2}\Omega_{3} \sqrt{\Omega_0\Omega_{4}} \big(a_2^{[2]} +p_2^{[2]}( t_{0} +t_{2} )  \big) \big(a_1^{[1]} +p_1^{[1]}( t_{2} -t_{3} )  \big)\big(a_1^{[1]} -p_1^{[1]}( t_{3} +t_{4} )  \big)}{64(2\ch+\Omega_{2}-\Omega_0)(3\ch-\Omega_{3}-\Omega_0)}   \, , \\
 \label{b1245p4bis} 
\betoneq{2}{2,3}{-,-}  & := \frac{\Omega_{2}\Omega_{3} \sqrt{\Omega_0\Omega_{4}}\big(a_2^{[2]} +p_2^{[2]}( t_{0} +t_{2} )  \big) \big(a_1^{[1]} +p_1^{[1]}( t_{2} +t_{3} )  \big)\big(a_1^{[1]} +p_1^{[1]}( t_{3} -t_{4} )  \big) }{64(2\ch+\Omega_{2}-\Omega_0)(3\ch+\Omega_{3}-\Omega_0)}   \, ,\\
 \label{b1245p4tris}
\betoneq{2}{2,3}{+,-}  & :=  \frac{\Omega_{2}\Omega_{3} \sqrt{\Omega_0\Omega_{4}} \big(a_2^{[2]} +p_2^{[2]}( t_{0} -t_{2} )  \big) \big(a_1^{[1]} +p_1^{[1]}(- t_{2}+ t_{3} )  \big)\big(a_1^{[1]} +p_1^{[1]}( t_{3} -t_{4} )  \big)}{64(2\ch-\Omega_{2}-\Omega_0)(3\ch+\Omega_{3}-\Omega_0)} \, , \\
 \label{b1245p4quat} 
\betoneq{2}{2,3}{+,+} & := \frac{\Omega_{2}\Omega_{3} \sqrt{\Omega_0\Omega_{4}} \big(a_2^{[2]} +p_2^{[2]}( t_{0} -t_{2} )  \big) \big(a_1^{[1]} -p_1^{[1]}( t_{2} +t_{3} )  \big)\big(a_1^{[1]} -p_1^{[1]}( t_{3} +t_{4} )  \big) }{64(2\ch-\Omega_{2}-\Omega_0)(3\ch-\Omega_{3}-\Omega_0)}  \, ,
\\
 \label{b3123p4}
\betoneq{3}{1,2,3}{+,+,+}  & :=  \frac{\Omega_{1}\Omega_{2}\Omega_{3} \sqrt{\Omega_0\Omega_{4}} \big(a_1^{[1]} +p_1^{[1]}( t_{0} - t_{1} )  \big) \big(a_1^{[1]} -p_1^{[1]}( t_{1} + t_{2} )  \big) \big(a_1^{[1]} -p_1^{[1]}( t_{2} + t_{3} )  \big)\big(a_1^{[1]} -p_1^{[1]}( t_{3} +t_{4} )  \big)}{256 (\ch-\Omega_{1}-\Omega_0) (2\ch-\Omega_{2}-\Omega_0)(3\ch-\Omega_{3}-\Omega_0)}   \, ,
\\
 \label{b13p4bis}
\betoneq{3}{1,2,3}{-,+,+}  & :=  \frac{\Omega_{1}\Omega_{2}\Omega_{3} \sqrt{\Omega_0\Omega_{4}} \big(a_1^{[1]} +p_1^{[1]}( t_{0} + t_{1} )  \big) \big(a_1^{[1]}+p_1^{[1]}( t_{1} - t_{2} )  \big) \big(a_1^{[1]} -p_1^{[1]}( t_{2} + t_{3} )  \big)\big(a_1^{[1]} -p_1^{[1]}( t_{3} +t_{4} )  \big)}{256 (\ch+\Omega_{1}-\Omega_0) (2\ch-\Omega_{2}-\Omega_0)(3\ch-\Omega_{3}-\Omega_0)} \, ,
\\
 \label{b13p4tris}
\betoneq{3}{1,2,3}{+,-,+}  & :=  \frac{\Omega_{1}\Omega_{2}\Omega_{3} \sqrt{\Omega_0\Omega_{4}} \big(a_1^{[1]} +p_1^{[1]}( t_{0} - t_{1} )  \big) \big(a_1^{[1]}-p_1^{[1]}( t_{1} - t_{2} )  \big) \big(a_1^{[1]} -p_1^{[1]}(   t_{3}- t_{2} )  \big)\big(a_1^{[1]} -p_1^{[1]}( t_{3} +t_{4} )  \big)}{256 (\ch-\Omega_{1}-\Omega_0) (2\ch+\Omega_{2}-\Omega_0)(3\ch-\Omega_{3}-\Omega_0)} \, ,
\\
 \label{b13p4quat}
\betoneq{3}{1,2,3}{+,+,-}  & :=  \frac{\Omega_{1}\Omega_{2}\Omega_{3} \sqrt{\Omega_0\Omega_{4}} \big(a_1^{[1]} +p_1^{[1]}( t_{0} - t_{1} )  \big) \big(a_1^{[1]}-p_1^{[1]}( t_{1} + t_{2} )  \big) \big(a_1^{[1]} -p_1^{[1]}( t_{2} - t_{3} )  \big)\big(a_1^{[1]} -p_1^{[1]}(t_{4}- t_{3} )  \big)}{256 (\ch-\Omega_{1}-\Omega_0) (2\ch-\Omega_{2}-\Omega_0)(3\ch+\Omega_{3}-\Omega_0)} \, ,
\\
 \label{b13p4pent}
\betoneq{3}{1,2,3}{-,-,+}  & :=  \frac{\Omega_{1}\Omega_{2}\Omega_{3} \sqrt{\Omega_0\Omega_{4}} \big(a_1^{[1]} +p_1^{[1]}( t_{0} + t_{1} )  \big) \big(a_1^{[1]}+p_1^{[1]}( t_{1} + t_{2} )  \big) \big(a_1^{[1]} +p_1^{[1]}( t_{2} - t_{3} )  \big)\big(a_1^{[1]} -p_1^{[1]}( t_{3} +t_{4} )  \big)}{256 (\ch+\Omega_{1}-\Omega_0) (2\ch+\Omega_{2}-\Omega_0)(3\ch-\Omega_{3}-\Omega_0)} \, ,
\\
 \label{b13p4pent}
\betoneq{3}{1,2,3}{+,-,-}  & :=  \frac{\Omega_{1}\Omega_{2}\Omega_{3} \sqrt{\Omega_0\Omega_{4}} \big(a_1^{[1]} +p_1^{[1]}( t_{0} - t_{1} )  \big) \big(a_1^{[1]}-p_1^{[1]}( t_{1} - t_{2} )  \big) \big(a_1^{[1]} +p_1^{[1]}(   t_{3}+t_{2} )  \big)\big(a_1^{[1]} -p_1^{[1]}(t_{4}- t_{3} )  \big)}{256 (\ch-\Omega_{1}-\Omega_0) (2\ch+\Omega_{2}-\Omega_0)(3\ch+\Omega_{3}-\Omega_0)} \, ,
\\
 \label{b13p4esa}
\betoneq{3}{1,2,3}{-,+,-}  & :=  \frac{\Omega_{1}\Omega_{2}\Omega_{3} \sqrt{\Omega_0\Omega_{4}} \big(a_1^{[1]} +p_1^{[1]}( t_{0} + t_{1} )  \big) \big(a_1^{[1]}-p_1^{[1]}(  t_{2}-t_{1} )  \big) \big(a_1^{[1]} -p_1^{[1]}( t_{2} - t_{3} )  \big)\big(a_1^{[1]} -p_1^{[1]}(t_{4}- t_{3} )  \big)}{256 (\ch+\Omega_{1}-\Omega_0) (2\ch-\Omega_{2}-\Omega_0)(3\ch+\Omega_{3}-\Omega_0)} \, ,
\\  \label{b3123p4fin}
\betoneq{3}{1,2,3}{-,-,-}  & :=  \frac{\Omega_{1}\Omega_{2}\Omega_{3} \sqrt{\Omega_0\Omega_{4}} \big(a_1^{[1]} +p_1^{[1]}( t_{0} + t_{1} )  \big) \big(a_1^{[1]}+p_1^{[1]}(t_{1}+  t_{2} )  \big) \big(a_1^{[1]} +p_1^{[1]}( t_{2} + t_{3} )  \big)\big(a_1^{[1]} +p_1^{[1]}( t_{3}-t_{4} )  \big)}{256 (\ch+\Omega_{1}-\Omega_0) (2\ch+\Omega_{2}-\Omega_0)(3\ch+\Omega_{3}-\Omega_0)} \, ,
\end{align}
\end{subequations}
with, in view \cite[(5.6)]{BCMV} 
and recalling the wavenumber $\uphi(4,\tth)$ from \cite[Lemma 2.7]{BCMV},
for $ j = 0,1,2,3,4 $,  
\begin{equation} \label{tj4}
  \Omega_j^{(4)} := \sqrt{\big(j+\uphi(4,\tth)\big)\tanh\Big(\tth\big(j+\uphi(4,\tth)\big)\Big)}
  \, ,\qquad 
  t_j^{(4)} = \sqrt{\frac{j+\uphi(4,\tth)}{\tanh\Big(\tth\big(j+\uphi(4,\tth)\big)\Big)}}\,  ,
  \end{equation} 
%\begin{equation} \label{tj4}
%\begin{alignedat}{2}
%&\Omega_0^{(4)} = \sqrt{\uphi(4,\tth)\tanh(\tth\uphi(4,\tth))}
 % \, ,\qquad &&t_0^{(4)}= \sqrt{\frac{\uphi(4,\tth)}{\tanh(\tth\uphi(4,\tth))}}\, , \\ 
 % &\Omega_1^{(4)} = \sqrt{\big(1+\uphi(4,\tth)\big)\tanh\Big(\tth\big(1+\uphi(4,\tth)\big)\Big)}
 % \, ,\quad  &&t_1^{(4)} = \sqrt{\frac{1+\uphi(4,\tth)}{\tanh\Big(\tth\big(1+\uphi(4,\tth)\big)\Big)}}\, , \\ 
 % &\Omega_2^{(4)} = \sqrt{\big(2+\uphi(4,\tth)\big)\tanh\Big(\tth\big(2+\uphi(4,\tth)\big)\Big)}
  %\, ,\quad &&t_2^{(4)}= \sqrt{\frac{2+\uphi(2,\tth)}{\tanh\Big(\tth\big(2+\uphi(2,\tth)\big)\Big)}}\, , \\
  %  &\Omega_3^{(4)} = \sqrt{\big(3+\uphi(4,\tth)\big)\tanh\Big(\tth\big(3+\uphi(4,\tth)\big)\Big)}
 % \, ,\quad &&t_3^{(4)}= \sqrt{\frac{3+\uphi(4,\tth)}{\tanh\Big(\tth\big(3+\uphi(4,\tth)\big)\Big)}}\, ,\\
%      &\Omega_4^{(4)} = \sqrt{\big(4+\uphi(4,\tth)\big)\tanh\Big(\tth\big(4+\uphi(4,\tth)\big)\Big)}
 % \, ,\quad &&t_4^{(4)}= \sqrt{\frac{4+\uphi(4,\tth)}{\tanh\Big(\tth\big(4+\uphi(4,\tth)\big)\Big)}}\, , 
  %\end{alignedat}
 % \end{equation}
and $a_\ell^{[\ell]}$ and $p_\ell^{[\ell]} $ in \eqref{ap12}  for $\ell=1,2$, in \eqref{ap3} for $ \ell=3$ and, in view \cite[(A.59)-(A.60)]{BMV_ed},
\begin{equation}
\label{ap4}
\begin{aligned}
p_4^{[4]} &= \frac{-\ch^{20}-39 \ch^{16}-366 \ch^{12}-850 \ch^8-657 \ch^4-135}{64 \ch^{19} (\ch^4+5)} \, ,\\
a_4^{[4]} &= \frac{9 \ch^{24}+238 \ch^{20}-233 \ch^{16}-1676 \ch^{12}+743 \ch^8-3042 \ch^4-135}{128 \ch^{20} (\ch^4+5)}\, .
\end{aligned}
\end{equation}
By  \eqref{dwch}  the terms in \eqref{ap4} admit the expansion, as $\tth\to +\infty$
\begin{equation}\label{dwap4}
    p_4^{[4]}=-\frac{16}{3}-\frac{605}{9}e^{-2\tth} +O\big(e^{-4\tth}\big) \,, \quad a_4^{[4]}=-\frac{16}{3}-\frac{698 }{9}e^{-2\tth} +O\big(e^{-4\tth}\big)\, . 
\end{equation}
For the terms in \eqref{tj4} we first need a slight improvement of formula \cite[(2.31)]{BCMV}. 
\begin{lem} The  wavenumber $\uphi(4,\tth)$ in {\em \cite[Lemma 2.7]{BCMV}} satisfies
\begin{equation}\label{expuphip=4}
\uphi(4,\tth) = \frac94 - \frac{15}2 e^{-2\tth}  + O\big(e^{-4\tth}\big)  \quad \text{as} \quad \tth \to +\infty \, .
\end{equation}
\end{lem}

\begin{proof} 
By \cite[formulas (A.11),(2.31)]{BCMV}, and the last line of the proof of \cite[Lemma 2.7]{BCMV} we have that 
$$
\uphi(4,\tth) = \frac94 + y_4(\tth) \quad \text{where} \quad 
y_4(\tth) = -\frac{15}{2}e^{-2\tth} + O\big(e^{-4\tth}\big) 
$$ 
which is \eqref{expuphip=4}.
\end{proof}

By \eqref{expuphip=4} we get 
\begin{equation}\label{tanhallp=4}
\tanh\big(\tth\big(j+\uphi(4,\tth)\big)\big) =  1 + O\big( e^{-4\tth}  \big)\, ,\qquad j=0,1,2,3\, .
\end{equation}
By \eqref{tanhallp=4} and \eqref{expuphip=4}, the terms in \eqref{tj3} admit the expansions, as $\tth \to +\infty $,
 \begin{equation} \label{tj4asymp}
\begin{alignedat}{2}
&\Omega_0^{(4)} 
= \tfrac32 - \tfrac52 e^{-2{\tth}} + O\big( e^{-4\tth}  \big)
  \, ,\quad &&t_0^{(4)}
  = \tfrac32 - \tfrac52 e^{-2{\tth}} + O\big(e^{-4\tth}\big)\, ,  \\ 
  &\Omega_1^{(4)} 
  = \tfrac{\sqrt{13}}{2} -  \tfrac{15}{2\sqrt{13}} e^{-2\tth} + O\big(e^{-4\tth}\big)
  \, ,\quad  &&t_1^{(4)}   = \tfrac{\sqrt{13}}{2} -  \tfrac{15}{2\sqrt{13}} e^{-2\tth} +  O\big(e^{-4\tth}\big)\, ,\\ 
    &\Omega_2^{(4)} 
    = \tfrac{\sqrt{17}}{2}-\tfrac{15 }{2 \sqrt{17}}e^{-2{\tth}} + O\big(e^{-4\tth}\big)
  \, ,\quad  &&t_2^{(4)}=  \tfrac{\sqrt{17}}{2}-\tfrac{15 }{2 \sqrt{17}}e^{-2{\tth}} + O\big(e^{-4\tth}\big) \, , \\ 
      &\Omega_3^{(4)} 
    =  \tfrac{\sqrt{21}}{2}-\tfrac{5\sqrt{3}}{2\sqrt{7}}   e^{-2{\tth}} + O\big(e^{-4\tth}\big)\,, \quad &&t_3^{(4)}  =\tfrac{\sqrt{21}}{2}-\tfrac{5\sqrt{3}}{2\sqrt{7}}   e^{-2{\tth}} + O\big(e^{-4\tth}\big)\,, \\ 
        &\Omega_4^{(4)} 
    =\tfrac{5}{2}-\tfrac{3 }{2}  e^{-2{\tth}} + O\big(e^{-4\tth}\big)
  \, , \quad
  &&t_4^{(4)}  = \tfrac{5}{2}-\tfrac{3 }{2}  e^{-2{\tth}}+ O\big(e^{-4\tth}\big) \, .
  \end{alignedat}
  \end{equation}
We now compute the asymptotic expansion of the summands in \eqref{b1partsp4}. We group them depending on the number of visited intermediate harmonics.\smallskip

\noindent{\bf No intermediate harmonics}. By \eqref{dwap4} and \eqref{tj4asymp} the first term in \eqref{b1partsp4} admits the  expansion
\begin{equation}\label{b0}
    b_0^{(4)}= -\frac{5 \sqrt{15} }{8} e^{-2\tth} + O\big(e^{-4\tth} \big) \quad \text{as} \quad \tth\to +\infty \, .
\end{equation}
Note that $ b_0^{(4)} $  vanishes in the deep water limit, as stated in \cite[(6.36)]{BCMV}.\smallskip

\noindent{\bf One intermediate harmonic}. We consider the terms in \eqref{b113p4}-\eqref{b115p4bis}. By \eqref{dwch}, \eqref{dwap12}, \eqref{dwap3}  and \eqref{tj4asymp} we get, as $\tth\to +\infty $, 
\begin{equation}
\begin{aligned}
    \betoneq{1}{1}{-} &=\frac{3 \sqrt{195}}{32}+\frac{1}{416} \left(1829 \sqrt{13}+2535\right) \sqrt{\frac{3}{5}} e^{-2\tth} + O\big(e^{-4\tth}\big)\, ,\\
    \betoneq{1}{1}{+} &= \frac{3 \sqrt{195}}{32}+\frac{1}{416} \left(1829 \sqrt{13}-2535\right) \sqrt{\frac{3}{5}} e^{-2\tth} + O\big(e^{-4\tth}\big)\, .
\end{aligned}
\end{equation}
Note that the difference 
\begin{equation}\label{b1}
    \betoneq{1}{1}{} := \betoneq{1}{1}{+}- \betoneq{1}{1}{-} = -\frac{39 \sqrt{15}}{16} e^{-2\tth} + O\big(e^{-4\tth}\big)
\end{equation}
vanishes in the deep-water limit as stated in \cite[Lemma 6.7]{BCMV}. Similarly
\begin{equation}
    \betoneq{1}{2}{-} =\frac{\sqrt{255}}{16}+\frac{4331 }{16 \sqrt{255}}e^{-2\tth} + O\big(e^{-4\tth}\big)\, , \quad
    \betoneq{1}{2}{+} = \frac{\sqrt{255}}{16}+\frac{4331 }{16 \sqrt{255}}e^{-2\tth} + O\big(e^{-4\tth}\big)\, ,
\end{equation}
whose difference 
\begin{equation}\label{b2}
    \betoneq{1}{2}{} := \betoneq{1}{2}{+}- \betoneq{1}{2}{-} =  O\big(e^{-4\tth}\big)
\end{equation}
vanishes in the deep-water limit as stated in \cite[Lemma 6.7]{BCMV}. Finally
\begin{equation}
\begin{aligned}
    \betoneq{1}{3}{-} &=\frac{9 \sqrt{35}}{32}+\Big(\frac{3723}{32 \sqrt{35}}-\frac{63 \sqrt{15}}{32}\Big) 
    e^{-2\tth} + O\big(e^{-4\tth}\big)\, ,\\
    \betoneq{1}{3}{+} &= \frac{9 \sqrt{35}}{32}+ \Big(\frac{3723}{32 \sqrt{35}}+\frac{63 \sqrt{15}}{32}\Big) 
    e^{-2\tth} + O\big(e^{-4\tth}\big)\, ,
\end{aligned}
\end{equation}
whose difference 
\begin{equation}\label{b3}
    \betoneq{1}{3}{} := \betoneq{1}{3}{+}- \betoneq{1}{3}{-} = 
    \frac{63 \sqrt{15} }{16} e^{-2\tth} + O\big(e^{-4\tth}\big)
\end{equation}
vanishes in the deep-water limit as stated in \cite[Lemma 6.7]{BCMV}. \smallskip

\noindent{\bf Two intermediate harmonics}. We now consider the terms in \eqref{b1234p4}-\eqref{b1245p4quat}. By \eqref{dwch}, \eqref{dwap12}, \eqref{dwap3}  and \eqref{tj4asymp} we get, as $\tth\to +\infty $, 
\begin{equation}
\begin{aligned}
    \betoneq{2}{1,2}{-,+} &=\frac{1}{256} \big(-17 \sqrt{195}-26 \sqrt{255}-11 \sqrt{3315} \, \big)\\
    &\qquad +\frac{\big(531828 \sqrt{13}-964171 \sqrt{17}-327976 \sqrt{221}+3174665 \big) }{113152 \sqrt{15}} e^{-2\tth} + O\big(e^{-4\tth}\big) \, ,\\
    \betoneq{2}{1,2}{-,-} &= \frac{1}{256} 
    \big(17 \sqrt{195}-26 \sqrt{255}-11 \sqrt{3315} \, \big)\\
    &\qquad -\frac{531828 \sqrt{13}+964171 \sqrt{17}+327976 \sqrt{221}+3174665}{113152 \sqrt{15}} e^{-2\tth} + O\big(e^{-4\tth}\big)\, ,\\
    \betoneq{2}{1,2}{+,-} &= \frac{1}{256} 
    \big(17 \sqrt{195}+26 \sqrt{255}-11 \sqrt{3315} \, \big) \\ 
    &\qquad+ \frac{-531828 \sqrt{13}+964171 \sqrt{17}-327976 \sqrt{221}+3174665}{113152 \sqrt{15}} e^{-2\tth} + O\big(e^{-4\tth}\big)\, ,\\
     \betoneq{2}{1,2}{+,+} &=  \frac{1}{256} \big(-17 \sqrt{195}+26 \sqrt{255}-11 \sqrt{3315} \, \big)\\
     &\qquad \frac{531828 \sqrt{13}+964171 \sqrt{17}-327976 \sqrt{221}-3174665}{113152 \sqrt{15}} e^{-2\tth} + O\big(e^{-4\tth}\big)\, ,
\end{aligned}
\end{equation}
whose signed sum 
\begin{equation}\label{b12}
    \betoneq{2}{1,2}{} :=  \betoneq{2}{1,2}{+,+}-\betoneq{2}{1,2}{-,+}+\betoneq{2}{1,2}{-,-}-\betoneq{2}{1,2}{+,-}  = -\frac{2873}{128}  \sqrt{\frac{5}{3}} e^{-2\tth} + O\big(e^{-4\tth}\big)
\end{equation}
vanishes in the deep-water limit as stated in \cite[Lemma 6.7]{BCMV}. Similarly we compute (via Mathematica)
% \begin{equation}
% \begin{aligned}
%     \betoneq{2}{1,3}{-,+} &=-\frac{3}{128} \big( 13 \sqrt{35}+7 \sqrt{195}+8 \sqrt{455}\, \big)\\
%     &\qquad +\frac{-538265 \sqrt{15}-473603 \sqrt{35}-206087 \sqrt{195}-165133 \sqrt{455}}{116480}e^{-2\tth} + O\big(e^{-4\tth}\big) \, ,\\
%     \betoneq{2}{1,3}{-,-} &= \frac{3}{128} \big(13 \sqrt{35}-7 \sqrt{195}+8 \sqrt{455}\, \big)\\
%     &\qquad +\frac{538265 \sqrt{15}-473603 \sqrt{35}+206087 \sqrt{195}-165133 \sqrt{455}}{116480}e^{-2\tth} + O\big(e^{-4\tth}\big)\, ,\\
%     \betoneq{2}{1,3}{+,-} &= -\frac{3}{128} \big(-13 \sqrt{35}-7 \sqrt{195}+8 \sqrt{455}\, \big)\\
%     &\qquad \frac{-538265 \sqrt{15}+473603 \sqrt{35}+206087 \sqrt{195}-165133 \sqrt{455}}{116480}e^{-2\tth} + O\big(e^{-4\tth}\big)\, ,\\
%      \betoneq{2}{1,3}{+,+} &=-\frac{3}{128} \big(-13 \sqrt{35}+7 \sqrt{195}+8 \sqrt{455}\, \big)\\
%      &\qquad +\frac{538265 \sqrt{15}+473603 \sqrt{35}-206087 \sqrt{195}-165133 \sqrt{455}}{116480}e^{-2\tth} + O\big(e^{-4\tth}\big)\, ,
% \end{aligned}
% \end{equation}
% whose signed sum 
\begin{equation}\label{b13}
    \betoneq{2}{1,3}{} :=  \betoneq{2}{1,3}{+,+}-\betoneq{2}{1,3}{-,+}+\betoneq{2}{1,3}{-,-}-\betoneq{2}{1,3}{+,-}  = \frac{1183 \sqrt{15}}{64}e^{-2\tth} + O\big(e^{-4\tth}\big)
\end{equation}
and
% \begin{equation}
% \begin{aligned}
%     \betoneq{2}{2,3}{-,+} &=-\frac{3}{256} \big(17 \sqrt{35}+14 \sqrt{255}+13 \sqrt{595}\, \big) \\
%     &\qquad  +\frac{920465 \sqrt{15}+340612 \sqrt{35}-498421 \sqrt{255}-376672 \sqrt{595}}{304640}e^{-2\tth} + O\big(e^{-4\tth}\big) \, ,\\
%     \betoneq{2}{2,3}{-,-} &=-\frac{3}{256} \big(17 \sqrt{35}-14 \sqrt{255}+13 \sqrt{595} \, \big) \\
%     &\qquad +\frac{-920465 \sqrt{15}+340612 \sqrt{35}+498421 \sqrt{255}-376672 \sqrt{595}}{304640}e^{-2\tth} + O\big(e^{-4\tth}\big)\, ,\\
%     \betoneq{2}{2,3}{+,-} &= \frac{1}{256} \big(51 \sqrt{35}+42 \sqrt{255}-39 \sqrt{595}\, \big)\\
%     &\qquad +\frac{920465 \sqrt{15}-340612 \sqrt{35}+498421 \sqrt{255}-376672 \sqrt{595}}{304640}e^{-2\tth} + O\big(e^{-4\tth}\big)\, ,\\
%      \betoneq{2}{2,3}{+,+} &=-\frac{3}{256} \big(-17 \sqrt{35}+14 \sqrt{255}+13 \sqrt{595}\, \big)\\
%      &\qquad +\frac{-920465 \sqrt{15}-340612 \sqrt{35}-498421 \sqrt{255}-376672 \sqrt{595}}{304640}e^{-2\tth} + O\big(e^{-4\tth}\big)\, ,
% \end{aligned}
% \end{equation}
% whose signed sum 
\begin{equation}\label{b23}
    \betoneq{2}{2,3}{} :=  \betoneq{2}{2,3}{+,+}-\betoneq{2}{2,3}{-,+}+\betoneq{2}{2,3}{-,-}-\betoneq{2}{2,3}{+,-}  =-\frac{1547 \sqrt{15}}{128}  e^{-2\tth} + O\big(e^{-4\tth}\big)\, ,
\end{equation}
which are both vanishing in the deep-water limit as stated in \cite[Lemma 6.7]{BCMV}.\smallskip

\noindent{\bf Three intermediate harmonics}. By \eqref{dwch}, \eqref{dwap12}, \eqref{dwap3}, \eqref{tj4asymp} and a very long computation using Mathematica, the terms in \eqref{b3123p4}-\eqref{b3123p4fin} are such that the term 
\begin{equation}\label{b123}
\begin{aligned}
    &\betoneq{3}{1,2,3}{} :=  \betoneq{3}{1,2,3}{-,+,-}-\betoneq{3}{1,2,3}{-,-,-}+\betoneq{3}{1,2,3}{+,-,-}-\betoneq{3}{1,2,3}{+,+,-} \\
    &\quad +\betoneq{3}{1,2,3}{+,+,+}-\betoneq{3}{1,2,3}{+,-,+}+\betoneq{3}{1,2,3}{-,-,+}-\betoneq{3}{1,2,3}{-,+,+}  =  O\big(e^{-4\tth}\big)
\end{aligned}
\end{equation}
vanishes in the deep-water limit as stated in \cite[Lemma 6.7]{BCMV}.

The sum of the terms in \eqref{b0}, \eqref{b1}, \eqref{b2}, \eqref{b3}, \eqref{b12}, \eqref{b13}, \eqref{b23} and \eqref{b123} gives the asymptotic expansion in \eqref{expbeta1p=4}. 
This concludes the proof of the  Theorem \ref{thm:main2new}.

%   e^{-2\tth} + O\big(e^{-4\tth}\big)

 \begin{footnotesize} 
 	
\end{footnotesize}

\noindent 
\footnotesize{Work  supported by PRIN 2020 (2020XB3EFL001) “Hamiltonian and dispersive PDEs”, PRIN 2022 (2022HSSYPN)   “TESEO-Turbulent Effects vs Stability in Equations from Oceanography".
P.\ Ventura is supported by the ERC STARTING GRANT 2021 ``Hamiltonian Dynamics, Normal Forms and Water Waves'' (HamDyWWa), Project Number: 101039762.
A.\ Maspero is supported by the ERC CONSOLIDATOR GRANT 2023 ``Generating Unstable Dynamics in dispersive Hamiltonian fluids'' (GUnDHam), Project Number: 101124921.
 Views and opinions expressed are however those of the authors only and do not necessarily reflect those of the European Union or the European Research Council. Neither the European Union nor the granting authority can be held responsible for them.
}
\normalsize

\medskip

\noindent 
{\it Massimiliano Berti}, SISSA, Via Bonomea 265, 34136, Trieste, Italy, \texttt{berti@sissa.it},  
\\[1mm]
{\it Livia Corsi}, Università degli Studi Roma Tre, Via Ostiense 133B, 00154, Roma, Italy, \texttt{liva.corsi@uniroma3.it},  
\\[1mm]
{   \it Alberto Maspero}, SISSA, Via Bonomea 265, 34136, Trieste, Italy,  \texttt{amaspero@sissa.it}, 
 \\[1mm] 
{   \it Paolo Ventura}, Università degli Studi di Milano, Via Saldini 50, 
20133 Milano, \texttt{paolo.ventura@unimi.it}.
\end{document}